\documentclass{article}
\usepackage[utf8]{inputenc}

\usepackage{authblk}

\usepackage{amsmath}
\usepackage{amsfonts}
\usepackage{mathrsfs}
\usepackage{mathtools}
\usepackage[dvipdfmx]{graphicx}
\usepackage{here}
\usepackage{amsthm}
\usepackage{algorithmic}
\usepackage{algorithm}
\usepackage{url}
\usepackage{stmaryrd}
\usepackage[title]{appendix}
\usepackage{amssymb}
\usepackage[all]{xy}
\usepackage{color}
\usepackage{extarrows}
\usepackage{multirow}
\usepackage{enumitem}

\newtheorem{theorem}{Theorem}[subsection]
\newtheorem{lemma}[theorem]{Lemma}
\newtheorem{proposition}[theorem]{Proposition}

\theoremstyle{definition}
\newtheorem{definition}[theorem]{Definition}
\newtheorem{example}[theorem]{Example}
\newtheorem{remark}[theorem]{Remark}

\theoremstyle{plain}
\newtheorem{theor}{Theorem}

\numberwithin{equation}{subsection}

\title{Explicit construction of a plane sextic model\\ for genus-five Howe curves, I}

\author[1]{Tomoki Moriya}
\author[2]{Momonari Kudo}
\affil[1]{\normalsize School of Computer Science, University of Birmingham}
\affil[2]{{\normalsize Faculty of Information Engineering, Fukuoka Institute of Technology}}

\date{\today}

\begin{document}

%\author[]{}
% \author[T. Moriya]{Tomoki Moriya}
% \author[M. Kudo]{Momonari Kudo}
% \address[T. Moriya, M. Kudo]{Department of Mathematical Informatics, The University of Tokyo, Japan}
% \email[T. Moriya, M. Kudo]{\{tomoki\_moriya,kudo\}@mist.i.u-tokyo.ac.jp}
%%\address[M. Kudo]{Department of Mathematical Informatics, The University of Tokyo, Japan}
%\email[T. Moriya]{\{tomoki\_moriya,\}@mist.i.u-tokyo.ac.jp}

%\keywords{Genus-$5$ curves, Howe curves, Jacobian varieties, Richelot isogenies, Superspecial curves}

\maketitle

\begin{abstract}
    In the past several years, {\it Howe curves} have been studied actively in the field of algebraic curves over fields of positive characteristic.
    Here, a Howe curve is defined as the desingularization of the fiber product over a projective line of two hyperelliptic curves.
    In this paper, we construct an explicit plane sextic model for non-hyperelliptic Howe curves of genus five.
    We also determine singularities of our sextic model.
    %develop an algorithm for enumerating superspecial non-hyperelliptic Howe curves of genus five.
\end{abstract}

%=====================
\section{Introduction}
%=====================

Let $k$ be an algebraically closed field of characteristic $p \neq 2$, and $\mathbb{P}^n$ the projective $n$-space over $k$.
A curve means a (possibly singular) projective variety of dimension one over $k$.
%Given a curve, constructing its explicit defining equation is a classically important problem, and this paper focuses 
As an important class of curves, this paper focuses on a fiber product of hyperelliptic curves.
Such a curve has a decomposable Jacobian (cf.\ \cite[Section 3.3]{glass2005hyperelliptic}), and is used for the construction of curves over finite fields with many rational points with respect to genus (cf.\ \cite{howe2016quickly}, \cite{howe2017curves}).
Howe presented methods to devise the construction for the case of genus $4$ in \cite{howe2016quickly}, and for the case of genera $5$, $6$, and $7$ in \cite{howe2017curves}.
In particular, Howe constructed a genus-$4$ curve as a fiber product of two genus-$1$ double covers sharing precisely one ramified point, and it is named {\it Howe} curve in \cite{kudo2020existence}.
Howe curves are also useful to produce superspecial/supersingular curves (cf.\ \cite{oort1991hyperelliptic}, \cite{kudo2020existence}, \cite{kudo2020algorithms}, \cite{ohashi2022fast}, \cite{moriya2022some}).
%, where they proved that there exists a supersingular Howe curve in arbitrary characteristic $p>3$.

In \cite{katsura2021decomposed}, Katsura-Takashima define a {\it generalized Howe curve} as a smooth curve birational to the fiber product of two hyperelliptic curves.
More precisely, let $C_1$ and $C_2$ be hyperelliptic curves of genera $g_1$ and $g_2$ over $k$ with $0<g_1 \leq g_2$, and let $\psi_1 : C_1 \to \mathbb{P}^1$ and $\psi_2 : C_2 \to \mathbb{P}^1$ be usual double covers (hyperelliptic structures).
Suppose also that there is no isomorphism $\varphi$ with $\psi_2 \circ \varphi = \psi_1$.
Then, the normalization $H$ of the fiber product $C_1 \times_{\mathbb{P}^1} C_2$ over these hyperellitic structures is called a generalized Howe curve associated with $C_1$ and $C_2$.
Denoting the genera of $C_1$ and $C_2$ by $g_1$ and $g_2$ respectively, and letting $r$ the number of ramification points in $\mathbb{P}^1$ common to $C_1$ and $C_2$, the genus of $H$ is $g = 2 (g_1+g_2)+1-r$, and $H$ of genus $g\geq 4$ is hyperelliptic if and only if $r=g_1+g_2+1$.
% \[
% \begin{aligned}
%     C_1 : y_1^2 = (x-\alpha_1) \cdots (x- \alpha_r) (x-\alpha_{r+1}) \cdots (x- \alpha_{2g+2}),\\
%     C_2 : y_1^2 = (x-\alpha_1) \cdots (x- \alpha_r) (x-\beta_{r+1}) \cdots (x- \beta_{2g+2}).
% \end{aligned}
% \]
% Then, the normalization of the fiber product $C_1 \times_{\mathbb{P}^1} C_2$ over $\mathbb{P}^1$ is called 
{
Howe curves of genera $3$ and $4$ have been studied, motivated by constructing superspecial or supersingular curves.
For instance, it is proved in \cite{kudo2020existence} that there exist supersingular Howe curves of genus $4$ in arbitrary characteristic $p > 3$.
Also in \cite{kudo2020algorithms}, \cite{ohashi2022fast}, and \cite{moriya2022some}, algorithms for enumerating superspecial Howe curves are proposed in the genus-$4$ non-hyperelliptic case, the genus-$4$ hyperelliptic case, and the genus-$3$ case, respectively.}

In this paper, we study genus-$5$ non-hyperelliptic Howe curves.
Such a curve is constructed from parameter sets $(g_1,g_2,r) = (2,2,4)$ or $(1,1,0)$.
This paper focuses on the former case (cf.\ the latter case will be studied in a separated paper~\cite{MK23-2}), so that we can take $C_1 : y_1^2 = x (x-1)(x-\alpha_1)(x-\alpha_2)(x-\alpha_3)$ and $C_2 : y_2^2 = x (x-1)(x-\alpha_1)(x-\beta_2)(x-\beta_3)$ with mutually distinct $\alpha_1,\alpha_2,\alpha_3,\beta_2, \beta_3 \in k \smallsetminus \{ 0, 1 \}$, see Subsection \ref{subsec:genus5} below.
In general, a non-hyperelliptic curve of genus $5$ is isomorphic to the complete intersection of three quadrics in $\mathbb{P}^4$ by canonical embedding.
% If it is trigonal, i.e., there is a degree-$3$ morphism to $\mathbb{P}^1$, then it is also birational to a quintic in $\mathbb{P}^1$ with a unique singularity.
% However, any Howe curve has automorphism group having a subgroup isomorphic to the Klein $4$-group $\mathbf{V}_4:=\mathbb{Z}/2\mathbb{Z} \times \mathbb{Z}/2\mathbb{Z} $, whence it cannot be trigonal, if the genus is equal to $1$ modulo $4$, see \cite[Corollary 3.5]{schweizer2015some}\footnote{\textcolor{red}{In \cite[Corollary 3.5]{schweizer2015some}, the characteristic $p$ of the base field is assumed to be $0$, but it can be proved that the same arguments hold if $p \gg 0$.}}.
%On the other hand, 
Kudo-Harashita proved in \cite{kudo2021parametrizing} that every non-hyperelliptic curve of genus $5$ is birational to a plane sextic singular curve, and provide an algorithm to compute a defining equation of the sextic for a given canonical model in $\mathbb{P}^4$.

A main result of this paper is to determine an equation defining such a sextic curve associated with a non-hyperelliptic Howe curve of genus $5$ defined by the parameter set $(g_1,g_2,r)=(2,2,4)$.
More precisely, we prove the following:

\begin{theor}\label{thm:main1}
    Every non-hyperelliptic Howe curve $H$ of genus five associated with two genus-$2$ curves $C_1$ and $C_2$ is birational to a plane sextic singular curve
    \[
    C: f = \sum_{0< i+j \leq 3} c_{2i,2j}Y^{2i}Z^{2j} = 0
    \]
    with a node $(0,0)$, where each coefficient $c_{2i,2j}$ is written as a polynomial in $\alpha_1, \alpha_2, \alpha_3, \beta_2, \beta_3$ explicitly.
    Moreover, once $\alpha_1, \alpha_2, \alpha_3, \beta_2,\beta_3$ are given, all the $9$ coefficients $c_{2i,2j}$ can be computed in a constant number of arithmetics over a field to which all $\alpha_i$'s and $\beta_j$'s belong.
\end{theor}

We also classify singularities of our plane sextic curve $C$ in Theorem \ref{thm:main1}.
Denoting by $\tilde{C}$ the projective closure in $\mathbb{P}^2$ of $C$, we obtain the following:

\begin{theor}
    The projective curve $\tilde{C}$ has exactly $3$ or $5$ singularities, all of which are of multiplicity $2$.
    Moreover, $\tilde{C}$ has $5$ double points generically.
\end{theor}

Thanks to our explicit equation together with obtained information of singularities, one can analyze non-hyperelliptic Howe curves of genus $5$ as singular plane curves.
For instance, one can determine the isomorphy of given two such curves, by computing their function fields.
This could also derive applications to enumerating curves defined over finite fields, e.g., superspecial curves.

% The rest of this paper is organized as follows:
% Section \ref{sec:pre} summarizes some known facts on Howe curves and curves of genus five.
% In Section \ref{sec:main}, we prove the main results;
% we explicitly construct a plane sextic singular curve birational to a non-hyperelliptic Howe curve of genus $5$, prove its absolute irreduciblity, and determine its singularities.\\

% \begin{theorem}
%     There exists an algorithm for generating superspecial non-hyperelliptic Howe curves of genus $5$ with complexity ??.
% \end{theorem}

% \begin{theorem}
%     Superspecial .... summarized in Table...
% \end{theorem}

% \noindent 
% \noindent \textcolor{red}{\bf Problems to be solved before publication (if possible)}:
% \begin{enumerate}
%     % \item \textcolor{red}{Only one of the cases (II) and (III) in Prop.\ \ref{prop:sing} holds? (we expect this is true.)}
%     \item A genus-$5$ non-hyperelliptic Howe curve is trigonal, or not?\\
%     \textcolor{red}{$\longrightarrow$ Not trigonal if \cite[Corollary 3.5]{schweizer2015some} holds even in characteristic $p>3$.} 
%     \item \textcolor{red}{The converse of the main theorem is true?
%     (Under reasonable assumptions, this would be true?)}
% \end{enumerate}

%=============================
\subsection*{Acknowledgements}
%=============================
The authors thank Kazuhiro Yokoyama for helpful comments on proving the irreducibility of parametric polynomials and on analyzing singularities of our plane sextic curve.
This work was supported by JSPS Grant-in-Aid for Young Scientists 20K14301 and 23K12949 and EPSRC through grant EP/V011324/1.

%========================
\section{Preliminaries}\label{sec:pre}
%========================

In this section, we first briefly review some known facts on generalized Howe curves and curves of genus five.
Subsequently, we also show some properties of factors of a bivariate sextic polynomial, which will be used in Subsection \ref{subsec:irr} below.
Throughout this section, let $k$ be an algebraically closed field of characteristic $p \geq 0$ with $p \neq 2$.

%==================================================================
\subsection{Generalized Howe curves}\label{subsec:generalHowe}
%==================================================================

In \cite{katsura2021decomposed}, Katsura-Takashima define a generalized Howe curve as follows:
Let $C_1$ and $C_2$ be hyperelliptic curves of genera $g_1$ and $g_2$ over $k$ with $0<g_1 \leq g_2$.
Explicitly, we can write
\[
\begin{aligned}
    C_1 : y_1^2 = \phi_1(x) = (x-\alpha_1) \cdots (x- \alpha_r) (x-\alpha_{r+1}) \cdots (x- \alpha_{2g_1+2}),\\
    C_2 : y_2^2 =\phi_2 (x)= (x-\alpha_1) \cdots (x- \alpha_r) (x-\beta_{r+1}) \cdots (x- \beta_{2g_2+2}),
\end{aligned}
\]
for pairwise distinct $\alpha_i, \beta_j \in k$ with $1 \leq i \leq 2g_1+2$ and $1 \leq j \leq 2g_2+2$.
%where $g_1$ and $g_2$ are genera of $C_1$ and $C_2$ respectively.
We can take $\alpha_i$ or $\beta_j$ to be $\infty$; in this case, we remove the factors $x-\alpha_i$ or $x-\beta_j$ from the equations for $C_1$ and $C_2$.
Let $\psi_1: C_1 \to \mathbb{P}^1$ and $\psi_2 : C_2 \to \mathbb{P}^1$ be the hyperelliptic structures, and assume that there is no isomorphism $\varphi$ with $\psi_2 \circ \varphi = \psi_1$.
In this case, the fiber product $C_1 \times_{\mathbb{P}^1} C_2$ over $\mathbb{P}^1$ is irreducible, and we have $r \leq g_1+g_2+1$.
Then, the normalization $H$ of the curve $C_1 \times_{\mathbb{P}^1} C_2$ is called a {\it generalized Howe curve} associated with $C_1$ and $C_2$.
Note that a defining equation for $H$ is locally given as $y_1^2 = \phi_1(x)$ and $y_2^2=\phi_2(x)$.
The genus $g(H)$ of $H$ is equal to $2(g_1+g_2) +1 - r$, see \cite[Proposition 1]{katsura2021decomposed}.
The hyperelliptic involutions of $C_1$ and $C_2$ lift to order-$2$ automorphisms $\sigma_1$ and $\sigma_2$ of $H$, and the quotient curve $C_3 := H/ \langle \sigma_1 \sigma_2 \rangle$ is given by
\[
C_3 : y_3^2 = \phi_3(x) = (x-\alpha_{r+1}) \cdots (x- \alpha_{2g_1+2}) (x-\beta_{r+1}) \cdots (x- \beta_{2g_2+2}),
\]
which has genus $g_3 = g_1+g_2+1-r$.
Katsura-Takashima also proved that, if $g(H) \geq 4$, then $H$ is hyperelliptic if and only if $r = g_1+g_2+1$, i.e., $C_3 \cong \mathbb{P}^1$.
%\begin{theorem}[{\cite[Proposition 1]{katsura2021decomposed}}]
%The genus $g(C)$ of $C$ is equal to $2(g_1 + g_2) + 1 − r$.
%\end{theorem}
% Note that a defining equation for $H$ is locally given as $y_1^2 = \phi_1(x)$ and $y_2^2=\phi_2(x)$.
%, and its homogenization is $y_1^2 z^{2g_1} = z^{2g_1+2}\phi_1(x/z)$ and $y_2^2 z^{2g_2} = z^{2g_2+2}\phi_2(x/z)$, while this model could have singularities.

%================================
\subsection{Defining equations for curves of genus five}\label{subsec:fact}
%================================

In general, curves of genus $5$ are classified first by the gonality into hyperelliptic, non-hyperelliptic and trigonal, and non-hyperelliptic and non-trigonal, where we used a fact that any hyperelliptic curve of genus $\geq 3$ cannot be trigonal.
A hyperelliptic curve of genus $5$ over $k$ is given by $y^2 = \phi(x)$ for a separable monic univariate polynomial $\phi(x) \in k[x]$ of degree $11$ with $9$ possibly non-zero coefficients, and a non-hyperelliptic curve of genus $5$ is canonically embedded into $\mathbb{P}^4$ as the complete intersection of $3$ quadratic surfaces.

Among non-hyperelliptic curves of genus $5$, every trigonal one is isomorphic to the normalization of a plane quintic with a unique singularity, and its explicit defining equation (in a reduced form) is provided in \cite[Sections 2 and 3]{kudo2022superspecial} for each type of the singularity; non-split node, split node, and cusp. 
As for non-hyperelliptic and non-trigonal ones, Kudo-Harashita proved in \cite{kudo2021parametrizing} that any such a curve $C$ of genus $5$ is birational to a plane singular sextic curve with only double points, and provided a method to explicitly construct such a sextic for a given canonical model of $C$.
More precisely, assuming $C$ is given as the complete intersection $Q_1 = Q_2 = Q_3 = 0$ in $\mathbb{P}^4$ for quadratic forms $Q_i$'s over $k$ in $5$ variables, one can compute a bivariate sextic polynomial $f(Y,Z)$ such that the projective closure in $\mathbb{P}^2$ of $f(Y,Z) = 0$ is birational to $C$.
A genus-$5$ non-hyperelliptic Howe curve $H$ of course has such a sextic model, but this cannot be constructed in a way similar to \cite{kudo2021parametrizing}, since $H$ is given as a fiber product, not as a canonical model.

\subsection{Some properties of factors of a bivariate sextic}
%===============================================================

Let $f(Y,Z)$ be a bivariate sextic polynomial of the form 
\[
f=b_{60}Y^6 + b_{42} Y^4Z^2  + b_{40} Y^4  + b_{24}Y^2 Z^4 + b_{22}Y^2 Z^2 + b_{20}Y^2  + b_{06}Z^6 + b_{04} Z^4 + b_{02}Z^2
\]
in the variables $Y$ and $Z$ over $k$.
Assume $b_{60},b_{06},b_{20},b_{02} \neq 0$ and $b_{20} = - b_{02}$.
% , and assume that $f$ is factored into $f = r H_1 H_2$ with $r=b_{60}$ for some $H_1,H_2 \in k[Y,Z]$ with $1 \leq \mathrm{deg}H_1, \mathrm{deg} H_2 \leq 5$ which are monic in $Y$.
% It is straightforward from the form of $f$ that we may suppose $(\mathrm{deg}H_1, \mathrm{deg} H_2) = (2,4)$ or $(3,3)$.
Since $f$ is a polynomial in $Y^2$ and $Z^2$, it satisfies the following:
\begin{itemize}
    \item[$(\ast)$] If $f$ is divisible by a polynomial $H(Y,Z)$ in $k[Y,Z]$, then it is also divisible by $H(-Y,Z)$, $H(Y,-Z)$, and $H(-Y,-Z)$. 
\end{itemize} 
In this subsection, we prove some properties of factors of $f$, by using $(\ast)$.

\begin{lemma}\label{lem:iired1}
    With notation as above, if $f$ has a linear factor, then it is divisible by $Y^2 - Z^2$.
    In this case, we have the following system of linear equations:
    \begin{equation}\label{eq:linear1}
            \left\{
    \begin{aligned}
    b_{60} + b_{42}+ b_{24} + b_{06} = 0,\\
    b_{40} + b_{22} + b_{04} = 0.
    \end{aligned}
    \right.
    \end{equation}
\end{lemma}

\begin{proof}
Assume that $f$ has a linear factor $H_1 \in k[Y,Z]$.
By our assumptions $b_{60} \neq 0$ and $b_{06} \neq 0$, we may suppose $H_1 = Y + a_1 Z + a_2$, where $a_1, a_2 \in k$ with $a_1 \neq 0$.
%Note that $a_1 \neq 0$, since the $Z^6$-coefficient $b_{06}$ in $f$ is not zero.
Then, $f$ is also divided by
\[
\begin{aligned}
    H_2 &:= - H_1(-Y,Z) = Y - a_1 Z - a_2,\\
    H_3 &:= H_1(Y,-Z) = Y - a_1 Z + a_2,\\
    H_4 &:= - H_1(-Y,-Z) = Y + a_1 Z - a_2.\\
\end{aligned}
\]
If $H_1$, $H_2$, $H_3$, and $H_4$ are all different, then 
\[
F = b_{60} H_1 H_2 H_3 H_4 Q
\]
for some quadratic polynomial $Q \in k[Y,Z]$ whose $Y^2$-coefficient is equal to $1$.
Here, we claim that $Q$ is reducible.
Indeed, if $Q$ is irreducible, then it follows from $(\ast)$ that $Q = Y^2 + a_3 Z^2 +a_4$ for some $a_3, a_4 \in k$ with $a_3 \neq 0$.
It also follows from $b_{20} \neq 0$ that $a_2 \neq 0$, whence $a_4 = 0$ by $b_{00}=0$.
This implies $Q = (Y + \sqrt{a_3} Z)(Y - \sqrt{a_3} Z)$, a contradiction.
Therefore, $Q$ is reducible, say $Q = (Y + a_3 Z + a_4)(Y+a_5 Z + a_6)$ for some $a_3,a_4,a_5,a_6 \in k$ with $a_3,a_5 \neq 0$.
By $b_{20} \neq 0$, we obtain $a_2 \neq 0$, and hence $a_4 = 0$ or $a_6=0$ from $b_{00} = 0$.
We may suppose $a_4=0$.
In this case, it follows from $(\ast)$ that $Y - a_3 Z = Y+ a_3 Z$ or $Y - a_3 Z = Y + a_5 Z + a_6$.
The former case is impossible, since $b_{06} \neq 0$ implies $a_3 \neq 0$.
Thus, the latter case holds, so that $a_5 = -a_3$ and $a_6 = 0$, say $Q = (Y+a_3 Z)(Y-a_3 Z)$.
Considering $b_{20} = - b_{02}$, we have $a_3^2 = 1$, so that $Q =Y^2 -Z^2$.
% $f$ is factored as
% \[
% f = (Y+Z)(Y-Z) (Y + a_1 Z + a_2) (Y- a_1 Z + a_2) (Y + a_1 Z - a_2) (Y- a_1 Z - a_2),
% \]
% where $a_1,a_2 \neq 0$, as desired.

Next, we consider the case where two of $H_1$, $H_2$, $H_3$, and $H_4$ are equal to each other.
Without loss of generality, we may suppose that $H_1$ is equal to one of the others.
In this case, it follows from $a_1 \neq 0$ that $H_1 \neq H_2$ and $H_1 \neq H_3$, whence $H_1 = H_4$.
Therefore, we have $a_2 = 0$, so that $f$ is divisible by $Y^2 - a_1^2 Z^2$.
A straightforward computation shows
\[
\begin{aligned}
    f =& (Y^2 - a_1^2 Z^2) Q \\
    &+(b_{60} a_1^6 + b_{42} a_1^4 + b_{24} a_1^2 + b_{06})Z^6 + (b_{40} a_1^4 + b_{22} a_1^2 +  b_{04}) Z^4 + b_{20}(a_1^2 - 1) Z^2,
\end{aligned}
\]
where
\[
Q = b_{60} Y^4 + (b_{60}a_1^2+ b_{42})Y^2 Z^2 + b_{40} Y^2 + (b_{60}a_1^4 + b_{42} a_1^2 + b_{24}) Z^4 + (b_{40}a_1^2 + b_{22}) Z^2 + b_{20} .
% a[6]+(a[3]*a[9]^2 + a[5])*Z^2 + a[1]*Y^4 + (a[1]*a[9]^2+a[2])*Y^2*Z^2 + \
% a[3]*Y^2 + (a[1]*a[9]^4 + a[2]*a[9]^2 + a[4])*Z^4
\]
% \[
% Q= Y^4 + (b_{42} + a_1^2)Y^2 Z^2 + b_{40} Y^2 + (b_{42}a_1^2 + b_{24} + a_1^4)Z^4 + (b_{40}a_1^2 + b_{22})Z^2 + b_{20}.
% \]
Since $b_{20} \neq 0$, one has $a_1^2 = 1$, and thus \eqref{eq:linear1} holds.
\end{proof}

\begin{lemma}\label{lem:irred2}
With notation as above, $f$ cannot be factored into the product of three irreducible quadratic polynomials in $k[Y,Z]$.    
\end{lemma}

\begin{proof}
Assume for a contradiction that $f = b_{60}Q_1 Q_2 Q_3$ for some irreducible $Q_i \in k[Y,Z]$ with $\mathrm{deg}(Q_i)=2$ such that the $Y^2$-coefficient in $Q_i$ is equal to $1$.
Since the constant term of $f$ is zero by our assumption, that of $Q_i$ is also zero for some $i$.
We may assume that the constant term of $Q_1$ is zero, say
\[
\left\{
\begin{aligned}
Q_1 & = Y^2 + (a_1 Z + a_2) Y + (a_3 Z^2 + a_4 Z),\\
Q_2 & = Y^2 + (a_5 Z + a_6) Y + (a_7 Z^2 + a_8 Z + a_{9}),\\
Q_3 & = Y^2 + (a_{10} Z + a_{11}) Y + (a_{12} Z^2 + a_{13} Z + a_{14}).
\end{aligned}
\right.
\]
Note that $a_9 = a_{14} = 0$ does not hold, since the $Y^2$-coefficient in $F$ is not zero.
Also by $(\ast)$, $F$ has an irreducible factor $Q_1 (-Y,Z)$, and therefore we may suppose $Q_1 = Q_1(-Y,Z)$ or $Q_2= Q_1(-Y,Z)$.

If $Q_1 = Q_1(-Y,Z)$, then $a_1=a_2=0$, so that $Q_1 = Y^2 + a_3 Z^2 + a_4 Z$.
Since $f$ is also divisible by $Q_1(Y,-Z)$, we have $Q_1 = Q_1(Y,-Z)$, $Q_2=Q_1(Y,-Z)$, or $Q_3 = Q_1(Y,-Z)$.
\begin{itemize}
    \item If $Q_1 = Q_1(Y,-Z)$, then $a_4 = 0$, whence we can factor $Q_1 = Y^2 + a_3 Z^2 = (Y + \sqrt{-a_3} Z) (Y - \sqrt{-a_3}Z)$.
    This contradicts the irreducibility of $Q_1$.
    \item For the other two cases, we may assume
    \[
    Q_2 = Q_1(Y,-Z) = Y^2 + a_3 Z^2 - a_4 Z.
    \]
    In this case, 
    %one can deduce $Q_3(Y,Z) = Q_3(-Y,Z) = Q_3(Y,-Z)$ by $a_{14} \neq 0$, so that $Q_3 = Y^2 + a_{12} Z^2 + a_{14}$.
    %However, 
    the $Y^2$-coefficient in $Q_1Q_2Q_3$ becomes zero, which contradicts our assumption $b_{20} \neq 0$.
\end{itemize}
Therefore, we may suppose $Q_2= Q_1(-Y,Z) = Y^2 -(a_1 Z + a_2) Y + (a_3 Z^2 + a_4 Z)$.
By $a_9 = 0$ but $a_{14} \neq 0$, one can also deduce $Q_3 = Q_3(-Y,Z) = Q_3(Y,-Z)$, so that $Q_3 = Y^2 + a_{12} Z^2 + a_{14}$.
Here, we have $Q_1 = Q_1(Y,-Z)$ or $Q_2 = Q_1(Y,-Z)$.
\begin{itemize}
    \item If $Q_1 = Q_1(Y,-Z)$, then $a_1 = a_4 = 0$, whence $Q_1 = Y^2 + a_2 Y + a_3 Z^2$ and $Q_2 = Y^2 - a_2 Y + a_3 Z^2$.
    \item If $Q_2 = Q_1(Y,-Z)$, then $a_2 = a_4 = 0$, so that $Q_1 = Y^2 + a_1 Z Y + a_3 Z^2$ and $Q_2 = Y^2 - a_1 Z Y + a_3 Z^2$.
\end{itemize}
In any case, the $Z^2$-coefficient in $Q_1 Q_2 Q_3$ is equal to zero, a contradiction.
\end{proof}

\begin{lemma}\label{lem:irred3}
    If $f$ is factored into $f = b_{60}H_1 H_2$, where $H_1$ and $H_2$ are quadratic and quartic irreducible polynomials in $k[Y,Z]$ respectively, then we can take
    \begin{equation}\label{eq:A2}
        \begin{cases}
        H_1 = Y^2 + (s Z^2 + t),\\
        H_2 = Y^4 + (u Z^2 + v)Y^2 + (w Z^4 - v Z^2)
    \end{cases}
    \end{equation}
    for some $s,t,u,v,w \in k$ with $s,t,v,w \neq 0$, whence
    \begin{eqnarray*}
        H_1H_2 \!\!\!&=& \!\!\!Y^6 + (s + u) Y^4 Z^2 + (t + v) Y^4 + (s u + w) Y^2 Z^4 \\
        && + (sv + t u - v) Y^2 Z^2 + t v Y^2 + s w  Z^6  + (-s v + t w) Z^4 - t v Z^2.
    \end{eqnarray*} %OK
\end{lemma}

\begin{proof}
    Since $H_i = H_i(-Y,Z) = H_i(Y,-Z)$ for each $i \in \{1,2\}$, we can write $H_1 = Y^2  + a_1 Z^2 + a_2$, and $H_2 = Y^4 + (a_3 Z^2 + a_4)Y^2 + (a_5 Z^4 + a_6 Z^2 + a_7)$.
    Since $H_1$ is irreducible, and since the constant term of $H_1H_2$ is zero, we have $a_2 \neq 0$ and $a_7 = 0$.
    Moreover, it follows from $b_{20} = -b_{02} \neq 0$ and $a_2 \neq 0$ that $a_6 = - a_4$.
    Expanding $H_1H_2$ together with $b_{60}, b_{20}, b_{02} \neq 0$, we obtain $a_1,a_4,a_5 \neq 0$, as desired.
\end{proof}

\begin{lemma}\label{lem:irred4}
    If $f$ is factored into the product of cubic irreducible polynomials $H_1$ and $H_2$ in $k[Y,Z]$, then we can take
    \begin{equation}\label{eq:B6}
        \begin{cases}
        H_1 = Y^3 + (sZ)Y^2 + (t Z^2 +u) Y + (v Z^3 + \varepsilon u Z),\\
        H_2 =  Y^3 - (sZ)Y^2 + (t Z^2 +u) Y - (v Z^3 + \varepsilon u Z)\\
    \end{cases}
    \end{equation}
    %{\bf (B6)} \ 
    for $s,t,u,v \in k$ with $u,v \neq 0$ and $\varepsilon = \pm 1$, whence
    \begin{eqnarray*}
            H_1 H_2 &=& Y^6 + (-s^2 + 2t) Y^4 Z^2 + 2u Y^4 + (-2sv + t^2) Y^2 Z^4 \\
            &&+ (-2\varepsilon su + 2tu) Y^2 Z^2 + u^2 Y^2 - v^2 Z^6 -2 \varepsilon u v Z^4 - u^2 Z^2.
    \end{eqnarray*}%OK
    Hence, the following system holds:
    \begin{equation}\label{eq:linear2}
        \left\{
\begin{aligned}
b_{40}^2 -4 b_{60} b_{20} = 0,\\
b_{04}^2 + 4 b_{06} b_{20} = 0.
\end{aligned}
\right.
    \end{equation}
\end{lemma}

\begin{proof}
We put
\[
\left\{
\begin{aligned}
H_1 = Y^3 + (a_1 Z + a_2)Y^2 + (a_3 Z^2 + a_4 Z + a_5)Y + (a_6 Z^3 + a_7 Z^2 + a_8 Z + a_9),\\
H_2 = Y^3 + (a_1' Z + a_2')Y^2 + (a_3' Z^2 + a_4' Z + a_5')Y + (a_6' Z^3 + a_7' Z^2 + a_8' Z + a_9').\\
\end{aligned}
\right.
\]
Note that $a_6$ and $a_6'$ are not zero, by our assumption $b_{06} \neq 0$.
From $(\ast)$ and the irredicibilities of $H_1$ and $H_2$, we have the following:
\begin{itemize}
    \item $H_1 = -H_1 (-Y,Z)$ or $H_2 = -H_1 (-Y,Z)$, and
    \item $H_1 = H_1 (Y,-Z)$ or $H_2 = H_1 (Y,-Z)$.
\end{itemize}
Among the $4$ cases, the cases $H_1 = -H_1 (-Y,Z)$ and $H_1 = H_1(Y,-Z)$ do not occur, since each of the $2$ cases implies $a_6=0$, a contradiction.
Therefore we obtain $H_2 = - H_1 (-Y,Z)$ and $H_2 = H_1(Y,-Z)$, and hence
\[
\left\{
\begin{aligned}
Y^3 + a_1 Z Y^2 + (a_3 Z^2 + a_5)Y + (a_6 Z^3  + a_8 Z),\\
Y^3 - a_1 Z Y^2 + (a_3 Z^2 + a_5)Y - (a_6 Z^3 + a_8 Z).\\
\end{aligned}
\right.
\]
Moreover, it follows from $b_{20} = - b_{02}$ that $a_8^2 = a_5^2$, as desired.
\end{proof}

%==============================================
\section{Sextic model for genus-5 Howe curves}\label{sec:main}
%==============================================

In this section, we shall provide an explicit plane sextic model for non-hyperelliptic Howe curves of genus five.

%==================================================
\subsection{Non-hyperelliptic Howe curves of genus 5}\label{subsec:genus5}
%==================================================

Throughout this section, let $H$ be a non-hyperelliptic Howe curve of genus five associated with hyperelliptic curves $C_1$ and $C_2$ of genus $2$ sharing exactly $3$ ramification points in $\mathbb{P}^1$, as described in Subsection \ref{subsec:generalHowe}.
We denote by $\iota_1$ and $\iota_2$ the hyperelliptic involutions of $C_1$ and $C_2$ respectively, and by $\sigma_1$ and $\sigma_2$ their lifts to $H$.
Then, the quotient curve $H/ \langle \sigma_1 \sigma_2 \rangle$ has genus one, and we denote it by $C_3$.
Here, transforming the $3$ ramification points to $0$, $1$, and $\infty$ by an element of $\mathrm{PGL}_2(k)$, we may assume that $C_1$, $C_2$, and $C_3$ are given as follows:
\begin{eqnarray*}
    C_1 & : & y_1^2 = \phi_1 (x) = x(x-1)(x-\alpha_1)(x-\alpha_2)(x-\alpha_3),\\
    C_2 & : & y_2^2 = \phi_2 (x) = x(x-1)(x-\alpha_1)(x-\beta_2)(x-\beta_3),\\
    C_3 & : & y_3^2 = \phi_3 (x) = (x-\alpha_2)(x-\alpha_3)(x-\beta_2)(x-\beta_3).
\end{eqnarray*}
where $\alpha_1$, $\alpha_2$, $\alpha_3$, $\beta_2$, and $\beta_3$ are pairwise distinct elements in $k \smallsetminus \{ 0, 1 \}$.

\begin{definition}\label{def:HoweType}
A quintuple $(\alpha_1,\alpha_2,\alpha_3,\beta_2,\beta_3) \in k^5$ is said to be {\it of Howe type} if $\alpha_1$, $\alpha_2$, $\alpha_3$, $\beta_2$, $\beta_3$ are pairwise distinct, and if any of them does not belong to $\{0,1\}$.
\end{definition}

In this setting, $H$ is determined uniquely (without considering isomorphisms) by the pair $(\alpha_1, \{ \{ \alpha_2, \alpha_3 \}, \{ \beta_2, \beta_3 \} \})$ such that $(\alpha_1,\alpha_2,\alpha_3,\beta_2,\beta_3)$ is a point in $k^5$ of Howe type.
When the characteristic of $k$ is positive, and when we take $\alpha_i$'s and $\beta_j$'s to be elements in a finite field $\mathbb{F}_q$, the number of $H$ is equal to
%When we take $\alpha_i$'s and $\beta_j$'s to be pairwise distinct elements in $\mathbb{F}_{p} \smallsetminus \{0,1\}$, the number of choices of $(\alpha_1, \{ \{ \alpha_2,\alpha_3 \}, \{ \beta_2, \beta_3 \} \})$ is
\[
(q-2) \times \binom{q-3}{4} \times \binom{4}{2} \times \frac{1}{2}= \frac{(q-2)(q-3)(q-4)(q-5)(q-6)}{8}.
\]

We also note that a genus-$1$ curve $E :y^2=(x-\alpha)(x-\beta)(x-\gamma)(x-\delta)$ is isomorphic to $E_{\lambda} : y^2=x(x-1)(x-\lambda)$ in Legendre form with $\lambda = \frac{( \beta - \gamma ) ( \delta - \alpha )}{(\beta - \alpha ) (\delta - \gamma)}$, whence $C_3$ is isomorphic to $E_{\lambda}$ with $\lambda = \frac{( \alpha_3 - \beta_2 ) ( \beta_3 - \alpha_2 )}{(\alpha_3 - \alpha_2 ) (\beta_3 - \beta_2)}$.

\subsection{Construction of our plane sextic model}\label{subsec:our sextic}
%========================

By setting
\[
\begin{aligned}
\phi &:= x(x-1)(x-\alpha_1) = x^3 - \sigma_1 x^2 + \sigma_2 x - \sigma_3,\\
Q_1 &:= (x-\alpha_2)(x-\alpha_3) = x^2 - \tau_1 x + \tau_2,\\
Q_2 &:= (x-\beta_2)(x-\beta_3) = x^2 - \rho_1 x + \rho_2,
\end{aligned}
\]
we can write $C_1 : y_1^2 = \phi Q_1$, $C_2 : y_2^2 = \phi Q_2$, and $C_3 : y_3^2 = Q_1Q_2$,
% For each integer $i$ with $1\leq i \leq 4$, we denote by $\sigma_i$ and $\tau_i$ the degree-$i$ elementary symmetric polynomial on $\alpha_1,\alpha_2,\alpha_3,\alpha_4$ and that on $\beta_1,\beta_2,\beta_3,\beta_4$ respectively, say
where
\[
\begin{array}{llll}
\sigma_1 := 1+\alpha_1, & \sigma_2 := \alpha_1, & \sigma_3 := 0, & \\
\tau_1 := \alpha_2 + \alpha_3, & \tau_2 := \alpha_2 \alpha_3, & \rho_1 := \beta_2 + \beta_3, & \rho_2 := \beta_2\beta_3.
\end{array}
\]
Putting $Y := \displaystyle \frac{y_1}{\phi}$ and $Z := \dfrac{y_2}{\phi}$ and squaring both sides, we obtain $\phi Y^2 = Q_1$ and $\phi Z^2 = Q_2$.
Here, we put
\[
\begin{aligned}
f_1 &:= \phi Y^2 - Q_1 = Y^2 x^3 - (1+ \sigma_1 Y^2) x^2 + (\tau_1 +\sigma_2 Y^2) x -(\tau_2 + \sigma_3 Y^2),\\
f_2 &:= \phi Z^2 - Q_2=Z^2 x^3 - (1+ \sigma_1 Z^2) x^2 + (\rho_1 +\sigma_2 Z^2) x - (\rho_2 + \sigma_3 Z^2),\\
f &:= \mathrm{Res}_x (f_1,f_2) \in k[Y,Z].
\end{aligned}
\]
Denoting by $|M|$ the determinant of a square matrix $M$, we have
\[
\begin{aligned}
    &f =
\begin{vmatrix}
Y^2 & -1- \sigma_1 Y^2 & \tau_1 +\sigma_2 Y^2 & -\tau_2 - \sigma_3 Y^2 & 0 & 0\\
0 & Y^2 & -1- \sigma_1 Y^2 & \tau_1 +\sigma_2 Y^2 & -\tau_2 - \sigma_3 Y^2 & 0\\
0 & 0 & Y^2 & -1- \sigma_1 Y^2 & \tau_1 +\sigma_2 Y^2 & -\tau_2 - \sigma_3 Y^2\\
Z^2 & -1- \sigma_1 Z^2 & \rho_1 +\sigma_2 Z^2 & -\rho_2 - \sigma_3 Z^2 & 0 & 0\\
0 & Z^2 & -1- \sigma_1 Z^2 & \rho_1 +\sigma_2 Z^2 & -\rho_2 - \sigma_3 Z^2 & 0\\
0 & 0 & Z^2 & -1- \sigma_1 Z^2 & \rho_1 +\sigma_2 Z^2 & -\rho_2 - \sigma_3 Z^2
\end{vmatrix}\\
&=\begin{vmatrix}
Y^2 & -1 & \tau_1-\sigma_1 & -\tau_2+\sigma_2+\sigma_1(\tau_1-\sigma_1)  & -\sigma_3-\sigma_2(\tau_1-\sigma_1) & \sigma_3(\tau_1-\sigma_1)\\
0 & Y^2 & -1 & \tau_1-\sigma_1 & -\tau_2+\sigma_2  & -\sigma_3\\
0 & 0 & Y^2 & -1 & \tau_1 & -\tau_2 \\
Z^2 & -1 & \rho_1-\sigma_1 & -\rho_2+\sigma_2+\sigma_1(\rho_1-\sigma_1) & -\sigma_3-\sigma_2(\rho_1-\sigma_1) & \sigma_3(\rho_1-\sigma_1)\\
0 & Z^2 & -1 & \rho_1-\sigma_1 & -\rho_2+\sigma_2 & -\sigma_3\\
0 & 0 & Z^2 & -1 & \rho_1 & -\rho_2 
\end{vmatrix}.
\end{aligned}
\]
This implies that $f$ is a polynomial in $Y^2$ and $Z^2$ with no constant term, and that its total degree is at most $6$.
Furthermore, the following lemma is obtained:
\begin{lemma}\label{lem:c60}
    Denoting by $c_{ij}$ the $Y^iZ^j$-coefficient in $f$, we have the following:
    \begin{enumerate}
        \item[(1)] $c_{60} = -\mathrm{Res}_x(\phi, Q_2)$ and $c_{06} = \mathrm{Res}_x(\phi, Q_1)$.
        \item[(2)] $c_{20} = -\mathrm{Res}_x(Q_1, Q_2)$ and $c_{02} = \mathrm{Res}_x(Q_1, Q_2)$.
    \end{enumerate}
    Hence, each of these $4$ coefficients is not zero for every quintuple $(\alpha_1,\alpha_2,\alpha_3,\beta_2,\beta_3)$ of Howe type.
    In particular, $f$ has total degree $6$.
\end{lemma}
\begin{proof}
   The $Y^6$-coefficient in $\mathrm{Res}_x(f_1,f_2)$ is
   \[
   \begin{vmatrix}
-\rho_2+\sigma_2+\sigma_1(\rho_1-\sigma_1) & -\sigma_3-\sigma_2(\rho_1-\sigma_1) & \sigma_3(\rho_1-\sigma_1)\\
 \rho_1-\sigma_1 & -\rho_2+\sigma_2 & -\sigma_3\\
 -1 & \rho_1 & -\rho_2 
\end{vmatrix},
   \]
   which is equal to 
   \[
   -
\begin{vmatrix}
1 & - \sigma_1 & \sigma_2 & - \sigma_3 & 0\\
0 & 1 & - \sigma_1 & \sigma_2 & - \sigma_3\\
1 & -\rho_1 & \rho_2 & 0 & 0 \\
0 & 1 & -\rho_1 & \rho_2 & 0 \\
0 & 0 & 1 & -\rho_1 & \rho_2 
\end{vmatrix}\\
= - \mathrm{Res}_x(\phi,Q_2) .
   \]
A similar computation shows $c_{06} = \mathrm{Res}_x (\phi, Q_1)$.
% The $Z^6$-coefficient is
% \[
% \begin{vmatrix}
% -\tau_2+\sigma_2+\sigma_1(\tau_1-\sigma_1)  & -\sigma_3-\sigma_2(\tau_1-\sigma_1) & \sigma_3(\tau_1-\sigma_1)\\
% \tau_1-\sigma_1 & -\tau_2+\sigma_2  & -\sigma_3\\
% -1 & \tau_1 & -\tau_2 
% \end{vmatrix} = \mathrm{Res}_x (\phi, Q_1)
% \]
Moreover, the $Y^2$-coefficient in $\mathrm{Res}_x(f_1,f_2)$ is computed as
\[
-
\begin{vmatrix}
 -1 & \tau_1-\sigma_1 & -\tau_2+\sigma_2  & -\sigma_3\\
 0 & -1 & \tau_1 & -\tau_2 \\
 -1 & \rho_1-\sigma_1 & -\rho_2+\sigma_2 & -\sigma_3\\
  0 & -1 & \rho_1 & -\rho_2 
\end{vmatrix}
= -\begin{vmatrix}
 1 & -\tau_1 & \tau_2  & 0\\
 0 & 1 & -\tau_1 & \tau_2 \\
 1 & -\rho_1 & \rho_2 & 0\\
  0 & 1 & -\rho_1 & \rho_2 
\end{vmatrix}
= -\mathrm{Res}_x (Q_1,Q_2).
\]
By a similar computation, we obtain $c_{02} = \mathrm{Res}_x (Q_1,Q_2)$, as desired.
% \[
% \begin{vmatrix}
%  -1 & \tau_1-\sigma_1 & -\tau_2+\sigma_2  & -\sigma_3\\
%  0 & -1 & \tau_1 & -\tau_2 \\
%  -1 & \rho_1-\sigma_1 & -\rho_2+\sigma_2 & -\sigma_3\\
%  0 & -1 & \rho_1 & -\rho_2 
% \end{vmatrix}
% = \begin{vmatrix}
%  1 & -\tau_1 & \tau_2  & 0\\
%  0 & 1 & -\tau_1 & \tau_2 \\
%  1 & -\rho_1 & \rho_2 & 0\\
%   0 & 1 & -\rho_1 & \rho_2 
% \end{vmatrix}
% = \mathrm{Res}_x (Q_1,Q_2)
% \]
\end{proof}

%========================
\subsection{Proof of Theorem \ref{thm:main1}}
%========================

The sextic $f(Y,Z)$ constructed in the previous subsection is absolutely irreducible, which we will prove in Subsection \ref{subsec:irr} below.
In this subsection, we shall prove that $C_1 \times_{\mathbb{P}^1} C_2$ is birationally equivalent to the plane sextic curve $C :f(Y,Z) = 0$, from which we conclude all the assertions of Theorem \ref{thm:main1}.
First, for any point $(x,y_1,y_2)$ on $C_1 \times_{\mathbb{P}^1} C_2$, it is straightforward that the corresponding point $(Y,Z)$ lies on $C$, whence we obtain a rational map:
\[
\Phi : C_1 \times_{\mathbb{P}^1} C_2 \dashrightarrow C \ ; \ (x,y_1,y_2) \mapsto \left( \frac{y_1}{\phi(x)}, \frac{y_2}{\phi(x)} \right),
\]
which is well-defined over all the points $(x,y_1,y_2)$ in $C_1 \times_{\mathbb{P}^1}C_2$ with $\phi(x) \neq 0$.
We can also construct the inverse rational map as follows:
For each point $(Y,Z)$ on $C$, it follows from $\mathrm{Res}_x (f_1(x,Y,Z),f_2(x,Y,Z)) = 0$ that there exists a common root of the univariate polynomials $f_1(x,Y,Z)$ and $f_2(x,Y,Z)$.
Choosing such a root $x$, we define a map
\[
C \longrightarrow C_1 \times_{\mathbb{P}^1} C_2 \ ; \ (Y,Z) \mapsto \left( x, \phi(x) Y, \phi(x) Z \right).
\]
If $x$ is unique, i.e., the gcd of $f_1(x,Y,Z)$ and $f_2(x,Y,Z)$ is linear in $x$, then clearly $x$ is represented as a rational function of $Y$ and $Z$.
%and moreover we can prove that the induced rational map $\Psi$ gives rise to the inverse of $\Phi$.
Here, we determine a sufficient condition on $Y$ and $Z$ such that $x$ is unique, by computing the gcd of $f_1(x,Y,Z)$ and $f_2(x,Y,Z)$ directly:
First, we divide $f_2$ by $f_1$, say
\[
r_1:=Y^2 f_2 - Z^2 f_1 = Z^2 Q_1 - Y^2 Q_2 = - (Y^2 - Z^2) x^2 + (\rho_1 Y^2 - \tau_1 Z^2)x - (\rho_2 Y^2 - \tau_2 Z^2)
\]
is the reminder.
Next, we divide $f_1$ by $r_1 = Z^2 Q_1 - Y^2 Q_2$:
\[
\begin{aligned}
      r_2:= & (Y^2 - Z^2) f_1 + Y^2 x (Z^2 Q_1 - Y^2 Q_2)\\
      = & ( - (Y^2-Z^2) (1 + \sigma_1 Y^2) + Y^2 (\rho_1 Y^2 - \tau_1 Z^2) ) x^2 \\
      & + ((Y^2 - Z^2)(\tau_1 + \sigma_2 Y^2) - Y^2 (\rho_2 Y^2 - \tau_2 Z^2)) x - (Y^2-Z^2)(\tau_2 + \sigma_3 Y^2)\\
      = & - ( (\sigma_1 - \rho_1)Y^4 - (\sigma_1-\tau_1) Y^2 Z^2 + Y^2 - Z^2 ) x^2 \\
      & + ( (\sigma_2 -\rho_2)Y^4 -(\sigma_2-\tau_2) Y^2 Z^2 + \tau_1 Y^2 - \tau_1 Z^2)x -(Y^2-Z^2)(\tau_2 + \sigma_3 Y^2),
\end{aligned}
\]
which is further reduced by $r_1$ into
\[
r_3 := (Y^2 - Z^2)r_2 - ( (\sigma_1 - \rho_1)Y^4 - (\sigma_1-\tau_1) Y^2 Z^2 + Y^2 - Z^2 ) r_1 .
\]
Here, $r_3$ is of the form $h_1 x + h_2$ with $h_1,h_2 \in k[Y,Z]$.
In particular, we can compute $h_1$ as
\[
\begin{aligned}
    h_1 = &  (Y^2 - Z^2) ( (\sigma_2 -\rho_2)Y^4 -(\sigma_2-\tau_2) Y^2 Z^2 + \tau_1 Y^2 - \tau_1 Z^2)\\
    & -  ( (\sigma_1 - \rho_1)Y^4 - (\sigma_1-\tau_1) Y^2 Z^2 + Y^2 - Z^2 )(\rho_1 Y^2 - \tau_1 Z^2)\\
    =& (-\sigma_1 \rho_1 + \sigma_2 + \rho_1^2 - \rho_2) Y^6 + (\sigma_1 \tau_1 + \sigma_1 \rho_1 - 2 \sigma_2 - 2 \tau_1 \rho_1 + \tau_2 + \rho_2) Y^4 Z^2  \\
    & + (\tau_1 - \rho_1) Y^4 + (-\sigma_1 \tau_1 + \sigma_2 + \tau_1^2 - \tau_2) Y^2 Z^4 + (-\tau_1 + \rho_1) Y^2 Z^2 .
\end{aligned}
\]
\if 0
\[
\begin{aligned}
    h_2 = & (Y^2 - Z^2)^2 (\tau_2 + \sigma_3 Y^2) \\
    &- ( (\sigma_1 - \rho_1)Y^4 - (\sigma_1-\tau_1) Y^2 Z^2 + Y^2 - Z^2 )(\rho_2 Y^2 - \tau_2 Z^2)\\
    =&  (-\sigma_1 \rho_2 + \sigma_3 + \rho_1 \rho_2) Y^6 + (\sigma_1 \tau_2 + \sigma_1 \rho_2 - 2 \sigma_3 - \tau_1 \rho_2 - \tau_2 \rho_1) Y^4 Z^2\\
    & + (\tau_2 - \rho_2) Y^4 + (-\sigma_1 \tau_2 + \sigma_3 + \tau_1 \tau_2) Y^2 Z^4 + (-\tau_2 + \rho_2) Y^2 Z^2
\end{aligned}
\]
\begin{eqnarray*}
%q_1 &:=& (-F_2)Y^6 - F_3 Y^4 Z^2 - F_1 Y^4 + (-F_4) Y^2 Z^4 + F_1Y^2 Z^2,\\
q_0 &= & -(\alpha_1 - \beta_2  - \beta_3 +1) \beta_2\beta_3 Y^6 + (\alpha_1 \alpha_2 \alpha_3 + \alpha_1 \beta_2 \beta_3 - \alpha_2 \alpha_3 \beta_2 - \alpha_2 \alpha_3 \beta_3\\
    & & + \alpha_2 \alpha_3 - \alpha_2 \beta_2 \beta_3 - \alpha_3 \beta_2 \beta_3 + \beta_2 \beta_3) Y^4 Z^2 + (\alpha_2 \alpha_3 - \beta_2 \beta_3)Y^4 \\
    & & - (\alpha_1 - \alpha_2 - \alpha_3 +1) \alpha_2 \alpha_3 Y^2 Z^4 - (\alpha_2 \alpha_3 - \beta_2 \beta_3) Y^2 Z^2
\end{eqnarray*}
\fi

\begin{lemma}\label{lem:h1}
    With notation as above, $h_1$ is not zero as a polynomial in $Y$ and $Z$, i.e., it does not hold
    \begin{eqnarray*}
\begin{cases}
    F_1 :=  -\sigma_1 \rho_1 + \sigma_2 + \rho_1^2 - \rho_2 = 0,\\
    F_2:= \sigma_1 \tau_1 + \sigma_1 \rho_1 - 2 \sigma_2 - 2 \tau_1 \rho_1 + \tau_2 + \rho_2= 0,\\
    F_3:=\tau_1 - \rho_1,\\
    F_4:=-\sigma_1 \tau_1 + \sigma_2 + \tau_1^2 - \tau_2 = 0.
\end{cases}
\end{eqnarray*}
\end{lemma}

\begin{proof}
Regarding $\alpha_1$, $\alpha_2$, $\alpha_3$, $\beta_2$, and $\beta_3$ as variables, we consider the ideal $I := \langle F_1,F_2,F_3,F_4 \rangle \subset k[\alpha_1,\alpha_2,\alpha_3,\beta_2,\beta_3]$.
It suffices to prove that the affine algebraic set $V(I)$ in $\mathbb{A}^5(k)$ has no element $(\alpha_1,\alpha_2,\alpha_3,\beta_2,\beta_3)$ such that $\alpha_1$, $\alpha_2$, $\alpha_3$, $\beta_2$, and $\beta_3$ are pairwise distinct.
One can check that $F_4 = F_3^2 - F_1 - F_2$, and thus $I = \langle F_1, F_2, F_3 \rangle$.
It also follows that $I$ contains an element
\begin{eqnarray*}
    G_2 &:=& -2 F_1+(\alpha_1+\alpha_3-2 \beta_2-2 \beta_3+1)F_3-F_2\\
    &=& \alpha_3^2 - \alpha_3 \beta_2 - \alpha_3 \beta_3 + \beta_2 \beta_3 \\
    &=&(\alpha_3-\beta_2)(\alpha_3-\beta_3).
\end{eqnarray*}
Therefore, we have $I = \langle F_1,G_2,F_3 \rangle$, and $G_2$ is not zero as long as $\alpha_3 \neq \beta_2$ and $\alpha_3 \neq \beta_3$, as desired.
\end{proof}

\begin{remark}
We heuristically found the polynomial $G_2$ in the proof of Lemma \ref{lem:h1} as an element of a Gr\"{o}bner basis of $I$; replacing $k$ by $\mathbb{Q}$, we computed the Gr\"{o}bner basis with Magma.
In fact, we can prove theoretically that $\{-F_1,G_2,F_3\}$ is a Gr\"{o}bner basis of $I$ with respect to the lexicographical order with $\alpha_1 > \alpha_2 > \alpha_3 > \beta_2 > \beta_3$, as follows:
The leading monomials of $-F_1$, $G_2$, and $F_3$ are respectively $\alpha_1 \beta_2$, $\alpha_3^2$, and $\alpha_2$, which are pairwise disjoint, i.e., they have no variable in common.
Buchberger's first criterion~\cite[Lemma 5.66]{Becker1993} implies that the $S$-polynomial of two polynomials having disjoint leading monomials is reduced into zero by themselves, and hence $\{ -F_1, G_2, F_3 \}$ is a (in fact the reduced) Gr\"{o}bner basis of $I$.
\end{remark}% 

By Lemma \ref{lem:h1}, it is straightforward that $f_1(x,Y,Z)$ and $f_2 (x, Y,Z)$ has a unique common root unless $h_1 (Y,Z) = 0$.
Moreover, we can prove the following:
%that the locus $h_1 (Y,Z) = 0$ is a proper subset of $C : f(Y,Z) = 0$.
%Indeed, if the two loci coincide, then  by Hilbert's Nullstellensatz
\begin{lemma}\label{lem:proper}
    The locus $h_1 (Y,Z) = 0$ in $C$ is a proper subset of $C$.
\end{lemma}

\begin{proof}
    If the locus coincides with $C$, it follows from Hilbert's Nullstellensat and the absolute irreducibility of $f$ that we can write $h_1 = c f$ for some $c \in k$.
    By Lemmas \ref{lem:h1} and \ref{lem:c60}, the $Z^6$-coefficient of $c f$ is not zero, whereas that of $h_1$ is zero, a contradiction. 
    % then $\sqrt{\langle h_1 \rangle} = \sqrt{\langle f \rangle}$ by Hilbert's Nullstellensatz, so that $h_1^m \in \langle f \rangle$.
    % Since $f$ is absolutely irreducible, it divides $h_1$.
    % Hence, we can write $h_1 = c f$ for some $c \in k \smallsetminus \{ 0 \}$, but this contradicts $c_{06}=0$.
\end{proof}

Note that the number of points at infinity of $C$ is finite. Thus, excluding the locus $h_1 (Y,Z) = 0$ and points at infinity from $C$, we obtain a rational map
\[
\Psi : C \dashrightarrow C_1 \times_{\mathbb{P}^1} C_2 \ ; \ (Y,Z) \mapsto \left( x, \phi(x) Y, \phi(x) Z \right)
\]
with $x = - h_2(Y,Z)/h_1(Y,Z)$, which is a unique common root of $f_1(x,Y,Z)$ and $f_2 (x, Y,Z)$.
%such a common root is unique, 
We claim that $\phi (x) \neq 0$ for each image $(x, \phi(x)Y, \phi(x) Z)$.
Indeed, if $\phi(x) = 0$, then it follows from $f_1(x,Y,Z) = f_2(x,Y,Z) = 0$ that $Q_1(x)=Q_2(x)=0$, a contradiction.
Therefore, $\Phi\circ\Psi$ is well-defined, and it is clearly equal to $\mathrm{\rm id}_C$ as rational maps.
%Here, it is easy to see that $\Phi\circ\Psi={\rm id}_{C}$ and $\Psi\circ \Phi = {\rm id}_{C_1 \times_{\mathbb{P}^1}C_2}$ as rational maps, whence we conclude the following:
Here, we also have the following:

\begin{theorem}
    With notation as above, $\Psi\circ \Phi$ is well-defined, and it is equal to $\mathrm{\rm id}_{C_1 \times_{\mathbb{P}^1}C_2}$ as rational maps.
    Hence, $C_1 \times_{\mathbb{P}^1} C_2$ and $C$ are birationally equivalent to each other.
\end{theorem}

\begin{proof}
    The numbers of points at infinity of $C_1 \times_{\mathbb{P}^1} C_2$ and $C$ are finite. Therefore, it suffices to consider the affine models of $C_1 \times_{\mathbb{P}^1} C_2$ and $C$ for proving the birationality of them.
    
    It suffices to show that there exists a point $(x,y_1,y_2) \in C_1 \times_{\mathbb{P}^1} C_2$ with $\phi(x) \neq 0$ such that $h_1(Y,Z) \neq 0$ for $Y =\displaystyle \frac{y_1}{\phi(x)}$ and $Z = \displaystyle\frac{y_2}{\phi(x)}$. 
    If so, $\Psi\circ \Phi$ is well-defined on a non-empty open set $(\Psi\circ\Phi)^{-1}(C_1 \times_{\mathbb{P}^1} C_2)$, and it follows from the uniqueness of the common root of $f_1(x,Y,Z) = f_2 (x,Y,Z) = 0$ that $\Psi\circ \Phi = {\rm id}_{C_1 \times_{\mathbb{P}^1}C_2}$.

    To prove the existence of such a point, we consider $\Phi$ as the composition of $\Phi_1 : C_1 \times_{\mathbb{P}^1} C_2 \dashrightarrow V(f_1,f_2)$ and $\Phi_2 : V(f_1,f_2) \dashrightarrow C$, where $V(f_1,f_2) \subset \mathbb{A}^3$.
    More precisely,
    \begin{itemize}
        \item $\Phi_1$ is defined at any point on $U_1 := \{ (x,y_1,y_2) \in C_1 \times_{\mathbb{P}^1} C_2:\phi(x) \neq 0 \}$ by $(x,y_1,y_2) \to (x,\frac{y_1}{\phi(x)},\frac{y_2}{\phi(x)})$, and
        \item $\Phi_2$ is defined at any point on $U_2:=\{(x,Y,Z) \in V(f_1,f_2) : (Y, Z) \neq (0,0)\}$ by $(x,Y,Z) \mapsto (Y,Z)$.
    \end{itemize}
    It is straightforward to see that $\Phi_1$ defines a bijective map from $U_1$ to $U_2$, and that $\Phi_2$ defines a surjective map from $U_2$ to $U_3 := \{ (Y,Z) \in C : (Y,Z) \neq (0,0) \}$.
    Since there are only finite points on $C$ satisfying $h_1=0$ by Lemma \ref{lem:proper}, it follows from the surjectiveness that there exists a desired point $(x,y_1,y_2)$.
\end{proof}

% \[
% \frac{y_1^2}{\phi(x)^2} = \frac{(x-\alpha_2)(x-\alpha_3)}{x(x-1)(x-\alpha_1)}
% \]
% \[
% \frac{y_2^2}{\phi(x)^2} = \frac{(x-\beta_2)(x-\beta_3)}{x(x-1)(x-\alpha_1)}
% \]
% \begin{proof}
%     It suffices to prove the locus $(h_1(Y,Z)=0) \cup (Z=0)$ in $C$ is a proper subset of $C$.
%     It follows from $c_{60}, c_{06} \neq 0$ by Lemma \ref{lem:c60} that both the loci $Y=0$ and $Z=0$ in $C$ are finite, whence we look at $(h_1(Y,Z) = 0) \cup \{ Y \neq 0 \} \cup \{ Z \neq 0 \}$.
% \end{proof}

%==================================================
\section{Irreducibility and singularity analysis}\label{sec:IrrSing}
%==================================================

We use the same notation as in the previous section.
In this section, we first prove that our sextic $f(Y,Z)$ constructed in the previous section is irreducible over the algebraically closed field $k$, with the help of some computer calculations.
Subsequently, we classify singularities of the projective closure in $\mathbb{P}^2 = \mathrm{Proj}(k[Y,Z,X])$ of $f(Y,Z) = 0$, where $X$ is an extra variable for homogenization.

%==================================================
\subsection{Absolute irreducibility of our sextic}\label{subsec:irr}
%==================================================

We denote by $c_{ij}$ the coefficient of $Y^iZ^j$ in $f(Y,Z)$.
Recall that our sextic $f(Y,Z)$ is of the form
\[
c_{60}Y^6 + c_{42} Y^4Z^2  + c_{40} Y^4  + c_{24}Y^2 Z^4 + c_{22}Y^2 Z^2 + c_{20}Y^2  + c_{06}Z^6 + c_{04} Z^4 + c_{02} Z^2,
\]
% (\pm y: 1: 0) in P^2(y,z,x)
% (Y^2 + **Z^2)*(Y^2 - ??Z^2)
where $c_{60} \neq 0$ for every quintuple $(\alpha_1,\alpha_2,\alpha_3,\beta_2,\beta_3)$ of Howe type.
%\textcolor{red}{Note that, regarding each $c_{ij}$ as a polynomial in $\alpha_1,\alpha_2,\alpha_3,\beta_2,\beta_3$, it is homogeneous of degree $i+j$ (\underline{to be checked})}.

Before discussing the absolute irreducibility of $f(Y,Z)$, we prove a lemma, which will be also used in Subsection \ref{subsec:sing} to analyze singularities.

\begin{lemma}\label{lem:p1p2}
With notation as above, we have the following:
\begin{enumerate}
    % \item[(1)] $c_{02} = - c_{20}$.
    % \item[(2)] $c_{20}, c_{06}, c_{02} \neq 0$ for every quintuple $(\alpha_1,\alpha_2,\alpha_3,\beta_2,\beta_3)$ of Howe type.
    \item[(1)] $c_{40}^2 - 4 c_{60} c_{20}$ and $c_{04}^2 + 4 c_{20}c_{06}$ are factored respectively into $(\beta_2-\beta_3)^2p_1^2$ and $(\alpha_2 - \alpha_3)^2 p_2^2$ as polynomials with respect to $\alpha_i$'s and $\beta_j$'s, where
    \begin{eqnarray*}
        p_1 \!\!\!&\!=\!&\!\!\!\alpha_1 \alpha_2 \alpha_3 \beta_2 + \alpha_1 \alpha_2 \alpha_3 \beta_3 - \alpha_1 \alpha_2 \alpha_3 - \alpha_1 \alpha_2 \beta_2 \beta_3 - \alpha_1 \alpha_3 \beta_2 \beta_3 + \alpha_1 \beta_2 \beta_3\\
        &&+ \alpha_2  \alpha_3 \beta_2^2 - \alpha_2 \alpha_3 \beta_2 \beta_3 + \alpha_2 \alpha_3 \beta_2 - \alpha_2 \alpha_3 \beta_3^2 + \alpha_2 \alpha_3 \beta_3 + \alpha_2 \beta^2 \beta_3\\
        &&+ \alpha_2 \beta_2 \beta_3^2 - \alpha_2 \beta_2 \beta_3 + \alpha_3 \beta_2^2 \beta_3 + \alpha_3 \beta_2 \beta_3^2 - \alpha_3 \beta_2 \beta_3 - \beta_2^2 \beta_3^2,\\
        p_2 \!\!\!&\!=\!&\!\!\!\alpha_1 \alpha_2 \alpha_3 \beta_2 + \alpha_1 \alpha_2 \alpha_3 \beta_3 - \alpha_1 \alpha_2 \alpha_3 - \alpha_1 \alpha_2 \beta_2 \beta_3 - \alpha_1 \alpha_3 \beta_2 \beta_3 + \alpha_1 \beta_2 \beta_3\\
        &&+ \alpha_2^2 \alpha_3^2 - \alpha_2^2 \alpha_3 \beta_2 - \alpha_2^2 \alpha_3 \beta_3 + \alpha_2^2 \beta_2 \beta_3 - \alpha_2 \alpha_3^2 \beta_2 - \alpha_2 \alpha_3^2 \beta_3 \\
        && + \alpha_2 \alpha_3 \beta_2 \beta_3 + \alpha_2 \alpha_3 \beta_2 + \alpha_2 \alpha_3 \beta_3- \alpha_2 \beta_2 \beta_3 + \alpha_3^2 \beta_2 \beta_3 - \alpha_3 \beta_2 \beta_3.
    \end{eqnarray*}
    Moreover, it follows that
    \[
    p_2 - p_1 = (\alpha_2 - \beta_2)(\alpha_2 - \beta_3) (\alpha_3 - \beta_2)(\alpha_3 - \beta_3)  ,
    \]
    so that at least one of $c_{40}^2 - 4 c_{60} c_{20}$ and $c_{04}^2 + 4 c_{20}c_{06}$ is non-zero for every quintuple $(\alpha_1,\alpha_2,\alpha_3,\beta_2,\beta_3)$ of Howe type.
    \item[(2)] $d_1 := c_{60} + c_{42} + c_{24} + c_{06}$ and $d_2 := c_{40} + c_{22} + c_{04}$ are factored into
    \[
    \begin{aligned}
       d_1 =& -(\alpha_2 \alpha_3 - \beta_2 \beta_3)(\alpha_2 \alpha_3 - \alpha_2 - \alpha_3 - \beta_2 \beta_3 + \beta_2 + \beta_3)\\
        & \cdot (\alpha_1 (\alpha_2 + \alpha_3 - \beta_2 - \beta_3) - \alpha_2 \alpha_3 + \beta_2 \beta_3),\\
        d_2 =& -(\alpha_2 - \beta_3) (\alpha_3 - \beta_2) (\alpha_2 - \beta_3) (\alpha_2 - \beta_2) (\alpha_2 + \alpha_3 - \beta_2 - \beta_3)
    \end{aligned}
    \]
    as polynomials with respect to $\alpha_i$'s and $\beta_j$'s.
    Moreover, at least one of $d_1$ and $d_2$ is non-zero for every quintuple $(\alpha_1,\alpha_2,\alpha_3,\beta_2,\beta_3)$ of Howe type.
\end{enumerate}
\end{lemma}

\begin{proof}
% Regarding $\alpha_1$, $\alpha_2$, $\alpha_3$, $\beta_2$, and $\beta_3$ as variable, a tedious computation with a computer shows that
% \begin{eqnarray*}
%     c_{20} &=& (\alpha_2 - \beta_2) (\alpha_2 - \beta_3)(\alpha_3 - \beta_2) (\alpha_3 - \beta_3)   ,\\
%     c_{06} &=& -\alpha_2 \alpha_3 (\alpha_2-1) (\alpha_3-1) (\alpha_1 - \alpha_2) (\alpha_1 - \alpha_3),\\
%     c_{02} &=& -(\alpha_3 - \beta_3) (\alpha_3 - \beta_2) (\alpha_2 - \beta_3) (\alpha_2 - \beta_2),
% \end{eqnarray*}
% whence (1) and (2) hold.
We can prove (1) and the first part of (2) also by a computer calculation.
To prove the second part of (2), assume for a contradiction that $d_1 = d_2  = 0$.
Since $(\alpha_1,\alpha_2,\alpha_3,\beta_2,\beta_3)$ is of Howe type, it follows from the factorization of $d_2$ that $\alpha_2 + \alpha_3 - \beta_2 - \beta_3 = 0$, and thus $d_1=(\alpha_2 \alpha_3 - \beta_2 \beta_3)^3 =0$.
Therefore, putting $\alpha = \alpha_2 + \alpha_3 = \beta_2 + \beta_3$ and $\beta = \alpha_2 \alpha_3 = \beta_2 \beta_3$, both the pairs $(\alpha_2,\alpha_3)$ and $(\beta_2,\beta_3)$ are the roots of $X^2 - \alpha X + \beta$, so that $\{\alpha_2, \alpha_3\} = \{ \beta_2,\beta_3\}$, a contradiction.
\end{proof}

By using Lemma \ref{lem:p1p2} together with Lemmas \ref{lem:iired1} -- \ref{lem:irred4}, we shall prove the absolute irreducibility of our sextic $f(Y,Z)$, in the following proposition:

\begin{proposition}\label{prop:irr}
The sextic $f(Y,Z)$ is irreducible over $k$.
\end{proposition}

\begin{proof}
    Assume for a contradiction that $f$ is reducible.
    If $f$ has a linear factor, then it follows from Lemma \ref{lem:iired1} that $c_{60} + c_{42} + c_{24} + c_{06} = 0$ and $c_{40} + c_{22} + c_{04}=0$, which contradicts Lemma \ref{lem:p1p2} (2).
    Therefore, $f$ has no linear factor.
    Also by Lemma \ref{lem:irred2}, it suffices to consider only the case where $f$ is factored into the product of two cubic irreducible polynomials, or quadratic and quartic ones.
    If $f$ is the product of two cubic irreducible polynomials, then Lemma \ref{lem:irred4} implies $ c_{40}^2 - 4 c_{60}c_{20} =0$ and $c_{04}^2 + 4 c_{06}c_{20} =0$, but this contradicts Lemma \ref{lem:p1p2} (1).
    % If $f(Y,Z)$ were reducible, we could apply Lemma \ref{lem:factor}:
    % We could write $f(Y,Z) = r H_1H_2$ with $r=c_{60}$, where $(H_1,H_2)\in k[Y,Z]^2$ is either of 12 types as in the lemma.
    % We consider the following two systems of equations:
    % \[
    % {\bf (E1)} \
    % \begin{cases}
    %     c_{40}^2 - 4 c_{60}c_{20} =0, \\
    %     c_{04}^2 + 4 c_{06}c_{20} =0.
    % \end{cases}
    % \quad
    % {\bf (E2)} \
    % \begin{cases}
    %     c_{24} + c_{06} + c_{42} + c_{60}=0, \\
    %     c_{22}+ c_{40} + c_{04}=0.
    % \end{cases}
    % \]
    % Then, as for 11 cases among the 12 cases listed in Lemma \ref{lem:factor}, we can easily check the following:
    % \begin{itemize}
    %     \item {\bf (A4)}, {\bf (A6)}, and {\bf (B1)} satisfy both {\bf (E1)} and {\bf (E2)}.
    %     \item {\bf (B6)} satisfies {\bf (E1)}.
    %     \item {\bf (A1)}, {\bf (A3)}, {\bf (A5)}, {\bf (B2)}, {\bf (B3)}, {\bf (B4)}, and {\bf (B5)} satisfy {\bf (E2)}.
    % \end{itemize}
    % However, these contradict Lemma \ref{lem:p1p2} (3) or (4).
    Here, we consider the remaining case, namely $f = r H_1 H_2$ as in \eqref{eq:A2} of Lemma \ref{lem:irred3}, so that
    \begin{eqnarray*}
        rH_1H_2 \!\!\!&=& \!\!\! r Y^6 + (s + u)r Y^4 Z^2 + (t + v)r Y^4 + (s u + w)r Y^2 Z^4 \\
        && + (sv + t u - v)r Y^2 Z^2 + t v r Y^2 + s w r Z^6  + (-s v + t w)r Z^4 - t v rZ^2
    \end{eqnarray*}
    with $r = c_{60}$ and $s,t,v,w \neq 0$.
    From expansion of $f = rH_1 H_2$, we can write $s, t, u,v,w$ as rational functions of $\alpha_{i}$'s and $\beta_j$'s as follows:
    \begin{itemize}
        \item From the coefficients of $Y^4$ and $Y^2$ in $r H_1 H_2$, the elements $tr$ and $v r$ are the roots of $X^2 -c_{40} X + c_{20} c_{60}$.
        Since we can take $(\beta_2 - \beta_3)p_1$ as a square of $c_{40}^2 - 4 c_{20}c_{60}$ by Lemma \ref{lem:p1p2} (1), the roots are given by $(c_{40} \pm (\beta_2 - \beta_3) p_1)/2$, so that
        \[
        t, v = \frac{c_{40} \pm (\beta_2 - \beta_3) p_1}{2 r}.
        \]
        \item Once $t$ and $v$ are specified, the elements $s$ and $w$ are also determined.
        Indeed, from the coefficients of $Z^6$, $Z^4$, and $Z^2$ in $r H_1 H_2$, the elements $t w r$ and $- s v r$ are the roots of $X^2 -c_{04} X - c_{20} c_{06}$.
        Since we can take $(\alpha_2 - \alpha_3)p_2$ as a square of $c_{04}^2 + 4 c_{20}c_{06}$ by Lemma \ref{lem:p1p2} (1), the roots are given by $(c_{04} \pm (\alpha_2 - \alpha_3) p_2)/2$, so that
        \[
        tw, -sv = \frac{c_{04} \pm (\alpha_2 - \alpha_3) p_2}{2r}.
        \]
        \item From the coefficients of $Y^4 Z^2$, $Y^2 Z^4$, and $Y^2Z^2$ in $r H_1 H_2$, the element $u$ is the common root of the linear system 
        \[
        \begin{cases}
        r u  = c_{42} - r s,\\
        r s u = c_{24}- r w,\\
        r t u = c_{22} - r s v + r v
        \end{cases}
        \]
        with respect to $u$; therefore, we have $\xi_1 := c_{24}- r w - s ( c_{42} - r s) = 0$ and $\xi_2:=c_{22} - r s v + r v - t (c_{42} - r s)=0$.
        % \[
        % \begin{cases}
        % \xi_1 := c_{24}- r w - s ( c_{42} - r s) = 0,\\
        % \xi_2:=c_{22} - r s v + r v - t (c_{42} - r s)=0.
        % \end{cases}
        % \]
    \end{itemize}
    Recall that $r,t,v \neq 0$, and we can compute $\xi_1$ as
    \[
    \begin{aligned}
        \xi_1 =& \frac{8 r^3 t v^2 c_{24} - 8 r^4 t v^2w - 8 r^3 s t v^2 c_{42} + 8r^4 s^2 t v^2 }{ 8 r^3 t v^2 }\\
        =& \frac{(2r t) (2 r v)^2 c_{24} - (2r t w) (2 rv)^2 r + (- 2 r s v) (2rt) (2 r v) c_{42} + (-2rsv)^2 (2r t)r }{ (2r t)(2 r v)^2 }.
    \end{aligned}
    \]
    Here, both the numerator and denominator of the right hand side are polynomials in $\alpha_i$'s and $\beta_j's$ over $k$, and we denote them by $N_{\xi_1}$ and $D_{\xi_1}$ respectively.
    Similarly, we can write $\xi_2$ as
    \[
    \begin{aligned}
        \xi_2 =& \frac{4 r^2 v c_{22} - 4r^3 s v^2 + 4r^3 v^2 - 4r^2 t v c_{42} + 4 r^3 s t v  }{ 4 r^2 v }\\
        =& \frac{2 r (2r v) c_{22} + r(-2rsv) (2r v) + r(2r v)^2 - (2rt)(2r v) c_{42} -r(-2rsv) (2 r t)  }{ 2 r(2rv) }.
    \end{aligned}
    \]
    and both the numerator and denominator of the right hand side are are polynomials in $\alpha_i,\beta_j$ over $k$.
    They are denoted by $N_{\xi_2}$ and $D_{\xi_2}$ respectively.

    Below we shall derive a contradiction from $\xi_1 = \xi_2=0$, by investigating factors of $N_{\xi_1}$ and $N_{\xi_2}$.
    Note that, by the values of $s$, $t$, $v$, and $w$, there are the following $4$ cases to be considered:
    \[
    {\bf (I)} \
    \begin{cases}
        2 r t = c_{40} + (\beta_2 - \beta_3)p_1,\\
        2 r v = c_{40} - (\beta_2 - \beta_3)p_1,\\
        2rtw  = c_{04} + (\alpha_2 - \alpha_3) p_2,\\
        -2rsv= c_{04} - (\alpha_2 - \alpha_3) p_2.
    \end{cases}
    \quad
    {\bf (II)} \
     \begin{cases}
        2 r t = c_{40} - (\beta_2 - \beta_3)p_1,\\
        2 r v = c_{40} + (\beta_2 - \beta_3)p_1,\\
        2rtw  = c_{04} + (\alpha_2 - \alpha_3) p_2,\\
        -2rsv= c_{04} - (\alpha_2 - \alpha_3) p_2.
    \end{cases}
    \]
    \[
    {\bf (III)} \
     \begin{cases}
        2 r t = c_{40} + (\beta_2 - \beta_3)p_1,\\
        2 r v = c_{40} - (\beta_2 - \beta_3)p_1,\\
        2rtw  = c_{04} - (\alpha_2 - \alpha_3) p_2,\\
        -2rsv= c_{04} + (\alpha_2 - \alpha_3) p_2.
    \end{cases}
    \quad
    {\bf (IV)} \
    \begin{cases}
        2 r t = c_{40} - (\beta_2 - \beta_3)p_1,\\
        2 r v = c_{40} + (\beta_2 - \beta_3)p_1,\\
        2rtw  = c_{04} - (\alpha_2 - \alpha_3) p_2,\\
        -2rsv= c_{04} + (\alpha_2 - \alpha_3) p_2.
    \end{cases}
    \]
    In each of the four cases, we can verify that $N_{\xi_1}$ and $D_{\xi_1}$ have a non-trivial common factor, and dividing both by their gcd (denoted by $G_{\xi_1}$), we confirmed by a computer calculation (in the case $p=0$) that
    \begin{eqnarray*}
        \xi_1 = -\frac{\eta_1 \eta_2 \eta_3}{D_{\xi_1}/G_{\xi_1}} \quad \mbox{and} \quad \xi_2 = -\frac{\eta_1' \eta_2' \eta_3'}{D_{\xi_2}/G_{\xi_2}}
    \end{eqnarray*}
    for some $\eta_1, \eta_2, \eta_3, \eta_1', \eta_2',\eta_3' \in R[\alpha_1,\alpha_2,\alpha_3,\beta_1,\beta_2]$ with $\eta_1'=p_1$, where $\eta_3'$ is of the form $\alpha_i - \beta_j$ with $i \neq j$, see Appendix \ref{app:eta} for details.
    Clearly this holds also for $p>2$.
    Since $\xi_1 = \xi_2 = 0$ with $\eta_3' \neq 0$, we have $\eta_1 \eta_2 \eta_3 = \eta_1' \eta_2'=0$, equivalently $\eta_i = \eta_j'=0$ for some $i,j$.
    For such a pair $(i,j)$, it follows that $\eta_i - \eta_j' = 0$.
    However, we can verify by a tedious computation that this does not hold for any $(i,j)$, see Appendix \ref{app:eta} for details on factorizations of $\eta_i - \eta_j'$.
    For example, in the case (I), we obtain $\eta_1 - \eta_1' = -\beta_3 (\alpha_3 - \beta_2) (\alpha_2 - \beta_3)(\alpha_2 - \beta_2)$, which is not zero, since $(\alpha_1,\alpha_2,\alpha_3,\beta_2,\beta_3)$ is of Howe type.
    Also for the other cases, similar factorizations of $\eta_i - \eta_j'$ are obtained, and theorefore the proof has been completed.
\end{proof}

\subsection{Analysis of singularities}\label{subsec:sing}
%=============================================

We use the same notation as in the previous subsections;
let $H$ be a non-hyperelliptic Howe curve of genus five over $k$ in characteristic $p$ with $p = 0$ or $p \geq 7$, and $C$ its associated sextic affine plane curve $f(Y,Z) = 0$ constructed in Subsection \ref{subsec:our sextic}.
The projective closure $\tilde{C}$ of $C$ in $\mathbb{P}^2 = \mathrm{Proj}(k[Y,Z,X])$ has arithmetic genus $g(\tilde{C})=10$ by degree-genus formula, and thus it has a singularity.
In general, the following inequality holds: % (cf.~\cite[]{}):
\begin{equation}\label{inequality:genus-multiplicity}
g(H) \le g(\tilde{C}) - \sum_{P\in \mathrm{Sing}(\tilde{C})} \frac{m_{P}(m_{P}-1)}{2},
\end{equation}
where $\mathrm{Sing}(\tilde{C})$ and $m_{P}$ respectively denote the set of singularities of $\tilde{C}$ in $\mathbb{P}^2$ and the multiplicity of $\tilde{C}$ at a point $P$.
Therefore, $\tilde{C}$ has at most five singular points.
Indeed, homogenizing $f$ as 
\[
\begin{aligned}
    F =& c_{60}Y^6 + c_{42} Y^4Z^2  + c_{24}Y^2 Z^4 + c_{06}Z^6 \\
    & + c_{40} Y^4 X^2 + c_{22}Y^2 Z^2 X^2 +  c_{04} Z^4 X^2 + c_{20}Y^2 X^4 + c_{02} Z^2 X^4,
\end{aligned}
\]
and denoting $ \frac{\partial F}{\partial Y}$, $ \frac{\partial F}{\partial Z}$, and $\frac{\partial F}{\partial X}$ by $F_Y$, $F_Z$, and $F_X$ respectively, one has
\[
\begin{aligned}
        F_Y =& 2 Y(3 c_{60} Y^4 + 2 c_{42} Y^2 Z^2 + 2 c_{40} Y^2 X^2 + c_{24} Z^4 + c_{22} Z^2 X^2+c_{20} X^4),\\
    F_Z =& 2 Z(c_{42} Y^4 + 2 c_{24} Y^2 Z^2 + c_{22} Y^2 X^2 + 3 c_{06} Z^4 + 2c_{04} Z^2X^2 + c_{02} X^4),\\
   F_X =& 2X(c_{40} Y^4 + c_{22} Y^2 Z^2 + 2 c_{20} Y^2 X^2 + c_{04} Z^4 + 2 c_{02} Z^2 X^2),
\end{aligned}
\]
whence $(0:0:1)$ is a singular point of $\tilde{C}$.

\begin{lemma}\label{lem:001}
   The singular point $(0:0:1)$ on $\tilde{C}$ is a node (i.e., an ordinary double point),
   and $\tilde{C}$ has exactly $2$ or $4$ singularities other than $(0:0:1)$.
\end{lemma}

\begin{proof}
Suppose that $(0:0:1)$ is the only singular point of $F=0$. Since $c_{20}=-c_{02}$ and $c_{20} \neq 0$ from Lemma \ref{lem:c60} (2), the polynomial $f$ is represented by
\[
f= c_{20}(Y-Z)(Y+Z)+(\text{higher terms}).
\]
Therefore, the point $(0:0:1)$ is a node of $F=0$. We have $g(H)=g(\tilde{C})-1$ from the genus-degree formula, see \cite[Chap.\ V, (3.9.2) and (3.9.5)]{hartshorne}.
This contradicts $g(H)=5$ and $g(\tilde{C})=10$.
%\end{proof}

%\textcolor{red}{It is straightforward that $(0:0:1)$ has multiplicity two}, and thus 
%It also follows from \eqref{inequality:genus-multiplicity} that any other singularity of $\tilde{C}$ is also a double point.
%Moreover, one can easily check that, if $(a:b:1)$ (resp.\ $(a:1:0)$) is a double point of $\tilde{C}$, then so is $(\pm a: \pm b : 1)$ (resp.\ $(\pm a : 1 : 0)$).
%Therefore, the number of singularities on $\tilde{C}$ is $3$ or $5$.
One can easily check that, if $(a:b:1)$ (resp.\ $(a:1:0)$) is a double point of $\tilde{C}$, then so is $(\pm a: \pm b : 1)$ (resp.\ $(\pm a : 1 : 0)$). Therefore, it follows from \eqref{inequality:genus-multiplicity} that any other singularity of $\tilde{C}$ is also a double point and the number of singularities on $\tilde{C}$ is $3$ or $5$.
\end{proof}

% \begin{lemma}
%     If $\tilde{C}$ has a singular point other than $(0:0:1)$, then it has exactly three or five double points.
% \end{lemma}
% (deg F) F = (dF/dY)*Y + (dF/dZ)*Z + (dF/dX)*X
% \begin{proof}
% if $(a: b : 1)$ and $(a : 1: 0)$ are singularities of $\tilde{C}$, then so are $(\pm a : \pm b : 1)$ and $()$
% \end{proof}

\begin{lemma}\label{lem:3sing1}
    The projective singular curve $\tilde{C}$ has a singular point of the form $(y:1:0)$ if and only if one of the following two systems holds:
    %  \[
    %  \!
    % (1)
    % \begin{cases}
    %     c_{24}= 0\\
    %     4c_{42}^3 + 27 c_{60}^2 c_{06} = 0
    % \end{cases}
    % \!\! \!\!
    % (2)
    %  \begin{cases}
    %     c_{42}= 0\\
    %     4c_{24}^3 + 27 c_{60} c_{06}^2 = 0
    % \end{cases}
    % \!\! \! \!
    %  (3)
    % \begin{cases}
    %     c_{42}^2 - 3 c_{60} c_{24} = 0\\
    %     c_{42} c_{24} - 9 c_{60}c_{06} = 0
    % \end{cases}\!
    % \]
   %  \[
   %  \begin{aligned}
   %    (1) & \
   % \begin{cases}
   %       c_{42}^2 - 3 c_{60} c_{24} \neq  0,\\
   %       4c_{42}c_{06}(c_{42}^2-3c_{60}c_{24})+ 4 c_{60}c_{24} (c_{24}^2 - 3 c_{42} c_{06})\\
   %       =(c_{42}c_{24} - 9 c_{60}c_{06})(c_{42}c_{24}+3c_{60}c_{06}).
   %  \end{cases}\\
   %  (2) & \
   %  \begin{cases}
   %       c_{42}^2 - 3 c_{60} c_{24} = 0,\\
   %       c_{42} c_{24} - 9 c_{60}c_{06} = 0.
   %  \end{cases}
   %  \end{aligned}
   %  \]
    \[
    (1) 
   \begin{cases}
         c_{42}^2 - 3 c_{60} c_{24} \neq  0,\\
         c_{42} c_{24} - 9 c_{60}c_{06} \neq 0,\\
         4(c_{42}c_{06}(c_{42}^2-3c_{60}c_{24})+ c_{60}c_{24} (c_{24}^2 - 3 c_{42} c_{06}))\\
         -(c_{42}c_{24} - 9 c_{60}c_{06})(c_{42}c_{24}+3c_{60}c_{06}) = 0.
    \end{cases}\!\!
    (2)
    \begin{cases}
         c_{42}^2 - 3 c_{60} c_{24} = 0,\\
         c_{42} c_{24} - 9 c_{60}c_{06} = 0.
    \end{cases}
    \]
    % \[
    % 108 c_{60}^2 c_{06}^2 - 36 c_{60}c_{42}c_{24}c_{06} + 4 c_{60}c_{24}^3 + 4c_{42}^3c_{06} =  c_{42}^2c_{24}^2 - 18 c_{60}c_{42}c_{24}c_{06} + 81 c_{60}^2c_{06}^2
    % \]
    In this case, only one of (1) and (2) is satisfied.
    % In each case, the singular points of $\tilde{C}$ other than $(0:0:1)$ are given as follows:
    % {\rm (1)} $( \pm \sqrt{-2c_{42}/3 c_{60}} : 1 :0)$,
    % {\rm (2)} $( \pm \sqrt{-3c_{60}/2 c_{24}} : 1 :0)$,
    % {\rm (3)} $( \pm \sqrt{-c_{24}/c_{42}}:1 :0)$.
    Furthermore, it follows that $y= \pm \sqrt{- (c_{42} c_{24} - 9 c_{60}c_{06})/2(c_{42}^2 - 3 c_{60} c_{24})}$ in case (1), and $y=\pm \sqrt{-c_{24}/c_{42}}$ in case (2).
\end{lemma}

\begin{proof}
It follows from $c_{60},c_{06} \neq 0$ that both $(1:0:0)$ and $(0:1:0)$ do not lie on $\tilde{C}$.
We have
     \[
     \begin{aligned}
       F(Y,Z,0) =& c_{60}Y^6 + c_{42} Y^4Z^2  + c_{24}Y^2 Z^4 + c_{06}Z^6,\\
      F_Y(Y,Z,0) =& 2 Y  (3 c_{60} Y^4 + 2 c_{42} Y^2 Z^2  + c_{24} Z^4),\\
     F_Z(Y,Z,0)=& 2Z  (3 c_{06} Z^4 + 2 c_{24} Y^2 Z^2  + c_{42} Y^4),
     \end{aligned}
     \]
     and thus
     \[
     \begin{aligned}
     F(Y,1,0) =&c_{60}Y^6 + c_{42} Y^4  + c_{24}Y^2 + c_{06}, \\
     F_Y(Y,1,0) =& 2 Y  (3 c_{60} Y^4 + 2 c_{42} Y^2  + c_{24}),\\
     F_Z(Y,1,0)=& 2 (c_{42} Y^4 + 2 c_{24} Y^2 + 3 c_{06}  ),
     \end{aligned}
     \]
     where $ 6 F (Y,1,0) - Y \cdot F_Y (Y,1,0) = F_Z (Y,1,0)$.
     It also follows that
     \[
         c_{42} F_Y (Y,1,0) - 3 c_{60}Y F_Z(Y,1,0) = 2Y ( 2(c_{42}^2 - 3c_{60}c_{24}) Y^2 + c_{42} c_{24} - 9 c_{60}c_{06}).
    \]
    
     Assume that $(y : 1 : 0)$ with $y \neq 0$ is a singularity of $\tilde{C}$.
     It suffices to consider when $F_Y(Y,1,0)$ and $F_Z(Y,1,0)$ has a common root.
     First, if $c_{42} = c_{24} = 0$, then there is no root of $ F_Z(Y,1,0)$ since $c_{06} \neq 0$.
     Thus, we have $c_{42} \neq 0$ or $c_{24} \neq 0$.
    If $c_{42}^2 - 3 c_{60} c_{24} \neq 0$, then $y^2 = - (c_{42} c_{24} - 9 c_{60}c_{06})/2(c_{42}^2 - 3 c_{60} c_{24})$, so that $c_{42} c_{24} - 9 c_{60}c_{06} \neq 0$.
    Hence it follows from $F_Y(y,1,0) = 0$ that
    \[
    \frac{1}{18c_{60}} F_Y(y,1,0) =  27 c_{60}^2 c_{06}^2 - 18 c_{60}c_{42}c_{24}c_{06} + 4 c_{60}c_{24}^3 + 4c_{42}^3c_{06} - c_{42}^2c_{24}^2=0,
    \]
    which is equivalent to the second equality of (1).
         If $c_{42}^2 - 3 c_{60} c_{24} = 0$, we also have $c_{42} c_{24} - 9 c_{60}c_{06} = 0$, whence $c_{42},c_{24} \neq 0$ (and thus $c_{24}^2 - 3 c_{06}c_{42}=0$) and $c_{42}F_Y(Y,1,0) = 3 c_{60} Y F_Z(Y,1,0)$.
         In this case, $F_Y(Y,1,0)$ has exactly two roots $y$ given by $y^2 = -c_{42}/3c_{60}=-c_{24}/c_{42}$.
     % \begin{itemize}
     %     \item Assume $c_{24} = 0$, but $c_{42} \neq 0$.
     %     In this case, there are at most two singularities.
     %     Then we have $3 c_{60} y^2 + 2 c_{42} = 0$, and it follows from $c_{60} \neq 0$ that $y^2 = -2c_{42}/3 c_{60}$.
     %     By substituting this into $F_Z(Y,1,0)$, one has $4c_{42}^3 + 27 c_{60}^2 c_{06} = 0$.
     %     Conversely, if $c_{24}= 4c_{42}^3 + 27 c_{60}^2 c_{06} = 0$, then $F_{Y} = F_Z=0$ has exactly two roots.
     %     \item Assume $c_{24} \neq 0$, but $c_{42} = 0$.
     %     Then we have $2 c_{24} y^2 + 3 c_{06} = 0$, and it follows from $c_{24} \neq 0$ that $y^2 = -3c_{06}/2 c_{24}$.
     %     By substituting this into $F_Z(Y,1,0)$, one has $4c_{24}^3 + 27 c_{60} c_{06}^2 = 0$.
     %     Conversely, if $c_{42}= 4c_{24}^3 + 27 c_{60} c_{06}^2 = 0$, then $F_{Y} = F_Z=0$ has exactly two roots.
     %     \item Assume $c_{24} \neq 0$ and $c_{42} \neq  0$.
     %     \[
     %     (2c_{42}^2 - 6 c_{60}c_{24}) y^2 + c_{42} c_{24} - 9 c_{60}c_{06} = 0.
     %     \]
     %     If $c_{42}^2 - 3 c_{60} c_{24} \neq 0$, then $y^2 = (-c_{42} c_{24} - 9 c_{60}c_{06})/(c_{42}^2 - 3 c_{60} c_{24})$
     %     Otherwise $c_{42}^2 - 3 c_{60} c_{24} = c_{42} c_{24} - 9 c_{60}c_{06} = 0$, so that $F_Y = u Y F_Z$ for some constant $u$, and $F_Y(Y,1,0)$ has exactly two roots $y$ given by $y^2 = -c_{24}/c_{42}$.
     %     Therefore,
     % \end{itemize}
     Therefore, the ``only-if''-part holds.
     
     The ``if''-part can be checked easily by computations similar to ones in the above proof of the ``only-if''-part.
\end{proof}

\begin{lemma}\label{lem:3sing2}
The projective singular curve $\tilde{C}$ has a singular point of the form $(y:0:1)$ with $y \neq 0$ or $(0:z:1)$ with $z \neq 0$ if and only if one of the following two systems holds:
     \[
   (1) \
   \begin{cases}
        c_{40} \neq 0,\\
        c_{40}^2 - 4 c_{60}c_{20} = 0.
    \end{cases}
    \qquad 
    (2) \
     \begin{cases}
        c_{04} \neq 0,\\
        c_{04}^2 - 4c_{06}c_{02}=0.
    \end{cases}
    \]
    Note that only one of (1) and (2) can be satisfied by Lemma \ref{lem:p1p2} (1).
    Moreover, we have $y= \pm \sqrt{-2c_{20}/c_{40}}$ in case (1), and $z=\pm \sqrt{-2c_{02}/c_{04}}$ in case (2).
\end{lemma}

\begin{proof}
We suppose that $\tilde{C}$ has a singular point $(y:0:1)$ with $y \in k \smallsetminus \{ 0 \}$.
Since we have $F_Y (Y,0,1)= 2 Y(3 c_{60} Y^4 + 2 c_{40} Y^2 +c_{20})$, $F_Z (Y,0,1) = 0$, and $F_X(Y,0,1) = 2Y^2(c_{40} Y^2 +2 c_{20} )$, 
% \[
% \begin{aligned}
%     F(Y,0,1) =& Y^2(c_{60}Y^4  + c_{40} Y^2  + c_{20} ),\\
%     F_Y (Y,0,1)=& 2 Y(3 c_{60} Y^4 + 2 c_{40} Y^2 +c_{20}),\\
%     F_Z(Y,0,1) =& 0,\\
%    F_X(Y,0,1) =& 2Y^2(c_{40} Y^2 +2 c_{20} ),
% \end{aligned}
% \]
it follows from $c_{20} \neq 0$ that $c_{40} \neq 0$ and $y^2 = -2c_{20}/c_{40}$.
Substituting this into $F_Y(Y,0,1)$, we obtain the equality $4c_{60} c_{20} - c_{40}^2= 0$.
Conversely, if $c_{40} \neq 0$ and $4 c_{60} c_{20} - c_{40}^2 = 0$, then the system $F_Y(Y,0,1) = F_X(Y,0,1)=0$ has exactly two solutions $(y : 0 : 1)$ with $y^2 = -2c_{20}/c_{40}$ (and thus $y \neq 0$ by Lemma \ref{lem:c60} (2)).
Similarly, one can prove that $(0:z:1)$ with $z \in k \smallsetminus \{ 0 \}$ is a singularity of $\tilde{C}$ if and only if $c_{04} \neq 0$ and $c_{04}^2 + 4c_{06}c_{20}=0$;
in this case, it follows from $F_X(0,z,1)=0$ that $z^2 = -2c_{02}/c_{04}$.
% Similarly, if $\tilde{C}$ is singular at $(0:b:1)$, then it follows from
% \[
% \begin{aligned}
%     F(0,Z,1) =& c_{06}Z^6 +  c_{04} Z^4 + c_{02} Z^2,\\
%         F_Y(0,Z,1) =& 0,\\
%     F_Z(0,Z,1) =& 2 Z(3 c_{06} Z^4 + 2c_{04} Z^2 + c_{02}),\\
%    F_X(0,Z,1) =& 2Z^2(c_{04} Z^2 + 2 c_{02}),
% \end{aligned}
% \]
% that $c_{04} \neq 0$, $y^2 = -2c_{02}/c_{04}$, and $3c_{04} (c_{04}^2 + 4c_{06}c_{20})=0$, where we used $c_{02} = - c_{20}$.
% The converse also holds clearly, and 
%The second part follows from Lemma \ref{lem:p1p2} (3); $4 c_{60} c_{20} - c_{40}^2 = 0$ and $c_{04}^2 + 4c_{06}c_{20})=0$ 
\end{proof}

\begin{lemma}\label{lem:sing infinity zero}
If the projective singular curve $\tilde{C}$ has a singular point of the form $(y:0:1)$ with $y\neq 0$ or $(0:z:1)$ with $z\neq 0$, then $\tilde{C}$ has exactly three double points: $(0:0:1)$ and $(\pm y: 0: 1)$, or $(0:0:1)$ and $(0:\pm z: 1)$.
\end{lemma}

\begin{proof}
We prove that a singular point of the form $(y:0:1)$ or $(0:z:1)$ is a point that needs at least two blowups to resolve its singularity ({e.g.,} a tacnode).
If it is true, then we obtain the statement of Lemma \ref{lem:sing infinity zero} from Lemma \ref{lem:001} and the genus-degree formula, see \cite[Chap.\ V, (3.9.2) and (3.9.5)]{hartshorne}.

First, suppose that $\tilde{C}$ has a singular point of the form $(y:0:1)$ with with $y \in k \smallsetminus \{ 0 \}$.
From Lemma \ref{lem:3sing2}, we have $y^2=-2c_{20}/c_{40}$, $c_{40}\neq 0$, and $c_{40}^2-4c_{60}c_{20}=0$. Note that $c_{02}= -c_{20}\neq 0$ from Lemma \ref{lem:c60} (2).
We denote by $d_{ij}$ the coefficient of $Y^iZ^j$ in $f(Y+y,Z)$.
From a direct computation, we have
% \begin{gather*}
%     d_{02}=\frac{c_{20}}{c_{40}^2}(4c_{20}c_{42}-2c_{22}c_{40}-c_{40}^2),\\
%     f(Y+y,Z)=d_{20}Y^2+d_{02}Z^2+d_{12}YZ^2+d_{30}Y^3+(\text{higher terms}).
% \end{gather*}
\begin{align*}
    d_{02}&=\frac{c_{20}}{c_{40}^2}(4c_{20}c_{42}-2c_{22}c_{40}-c_{40}^2),\\
    f(Y+y,Z)&=d_{20}Y^2+d_{02}Z^2+d_{12}YZ^2+d_{30}Y^3+(\text{higher terms}).
\end{align*}
It follows from $c_{40}^2-4c_{60}c_{20}=0$ and Lemma \ref{lem:p1p2} (1) that $p_1=0$.
A computation shows that $4c_{20}c_{42}-2c_{22}c_{40}-c_{40}^2$ can be divided by $p_1$ as polynomials in $\alpha_i$'s and $\beta_j$'s, whence $d_{02}=0$.
Therefore, we have
\[
f(Y+y,Z)=d_{20}Y^2+d_{12}YZ^2+d_{30}Y^3+d_{40}Y^4+d_{22}Y^2Z^2+d_{04}Z^4+(\text{higher terms}).
\]
Since $(y:0:1)$ is a double point, it follows that $d_{20}\neq 0$.
Hence, it holds that
\[
k[[Y,Z]]/(f(Y+y,Z))\cong k[[Y,Z]]/(Y^2-Z^\alpha)
\]
for some integer $\alpha$ with $\alpha \geq 4$, where $k[[Y,Z]]$ is the ring of formal power series over $k$.
From \cite[Chap.\ V, (3.9.5)]{hartshorne}, we need at least two blowups to resolve the singularity of $(y:0:1)$.

Next, consider the case where $\tilde{C}$ has a singular point of the form $(0:z:1)$ with $z \in k \smallsetminus \{ 0 \}$.
In this case, recall from Lemma \ref{lem:3sing2} that $z^2=-2c_{02}/c_{04}$, $c_{04}\neq 0$, and $c_{04}^2+4c_{06}c_{20}=0$. 
We denote by $d_{ij}'$ the coefficient of $Y^iZ^j$ in $f(Y,Z+z)$.
A computation shows that
\begin{align*}
    d_{20}'&=\frac{c_{20}}{c_{04}c_{06}}(c_{04}c_{06}+2c_{06}c_{22}-c_{04}c_{24}),\\
    f(Y,Z+z)&=d_{20}'Y^2+d_{02}'Z^2+d_{21}'Y^2Z+d_{03}'Z^3+(\text{higher terms}).
\end{align*}
By $c_{04}^2+4c_{06}c_{20}=0$ together with Lemma \ref{lem:p1p2} (1), one has $p_2=0$.
One can also check that $c_{04}c_{06}+2c_{06}c_{22}-c_{04}c_{24}$ can be divided by $p_2$ by a direct calculation, and therefore $d_{20}'=0$.
Hence, from a similar discussion as in the previous paragraph, we obtain
\[
k[[Y,Z]]/(f(Y,Z+z))\cong k[[Y,Z]]/(Y^\alpha-Z^2)
\]
for some integer $\alpha$ with $\alpha \geq 4$, so that we need at least two blowups to resolve the singularity $(0:z:1)$, as desired.
\end{proof}

\begin{lemma}\label{lem:sing infinity zero2}
If the projective singular curve $\tilde{C}$ has a singular point of the form $(y:1:0)$ with $y\neq 0$, then $\tilde{C}$ has exactly three double points: $(0:0:1)$ and $(\pm y: 1:0)$.
\end{lemma}

\begin{proof}
From Lemma \ref{lem:001}, the number of singular points of $\tilde{C}$ is less than or equal to $5$. Therefore, the curve $\tilde{C}$ does not have any singular point of the form $(a:b:1)$ with $a,b\neq 0$ if it has a singular point of the form $(y:1:0)$.
Moreover, from Lemma \ref{lem:sing infinity zero}, the curve $\tilde{C}$ does not have any singular point of the form $(y:0:1)$ or $(0:z:1)$.
The remaining case is the case where $\tilde{C}$ has $4$ singular points of the form $(y:1:0)$, but this does not happen by Lemma \ref{lem:3sing2}.
% However, this case can be rejected because Lemma \ref{lem:3sing2} provides the explicit value of $\pm y$.
%This completes the proof.
\end{proof}

Putting Lemmas \ref{lem:001}, \ref{lem:3sing1},  \ref{lem:3sing2}, \ref{lem:sing infinity zero}, and \ref{lem:sing infinity zero2} together, we obtain the following:

\begin{proposition}\label{prop:sing}
    For the projective sextic plane curve $\tilde{C}$, only one of the following three cases is satisfied:
    \begin{enumerate}
        \item[(I)] $\tilde{C}$ has exactly five double points: $(0:0:1)$ and $(\pm y: \pm z:1)$,
        \item[(II)] $\tilde{C}$ has exactly three double points: $(0:0:1)$ and $(\pm y:1:0)$,
        \item[(III)] $\tilde{C}$ has exactly three double points: $(0:0:1)$ and $(\pm y: 0: 1)$, or $(0:0:1)$ and $(0:\pm z: 1)$,
    \end{enumerate}
    for some $y,z \in k \smallsetminus \{ 0 \}$.
    Moreover, the case (I) holds generically.
\end{proposition}

We confirmed with Magma that, when $p \geq 7$, the projective closure $\tilde{C}$ associated with a point of Howe type has in most cases exactly five double points, as $p$ becomes greater.
See Table \ref{table:sing} for computational results for $7 \leq p \leq 23$.
In Example \ref{ex:sing}, we show an example for each of the cases (I), (II), and (III).

\begin{center}
    \begin{table}[h]
    \centering
    \caption{The number of plane sextic curves $\tilde{C}$ associated to Howe triples $(\alpha_1, \{ \{ \alpha_2,\alpha_3 \}, \{ \beta_2, \beta_3 \} \})$ in Type $(m,n)$, where $\alpha_i,\beta_j \in \mathbb{F}_p$ for $1\leq i \leq 3$ and $2 \leq j \leq 3$, and where ``Type $(m,n)$'' means $\tilde{C}$ has $m$ double points of the form $(a:b:1)$ and $n$ ones of the form $(a:1:0)$, see also Proposition \ref{prop:sing}.}\label{table:sing}
    \vspace{2mm}
    \begin{tabular}{|c|rc|rc|rc|}
    \hline
         $p$ & \multicolumn{2}{c|}{Type $(5,0)$} &  \multicolumn{2}{c|}{Type $(1,2)$} &  \multicolumn{2}{c|}{Type $(3,0)$} \\
         \hline
          $7$ &      3 & (20\%) &    12 & (80\%) &     0 & (0\%)  \\ \hline
         $11$ &    846 & (45\%) &   720 & (38\%) &   324 & (17\%) \\ \hline
         $13$ &   3768 & (54\%) &  2100 & (30\%) &  1062 & (15\%) \\ \hline
         $17$ &  29925 & (66\%) &  9702 & (22\%) &  5418 & (12\%) \\ \hline
         $19$ &  65322 & (70\%) & 17472 & (19\%) & 10026 & (11\%) \\ \hline
         $23$ & 232065 & (76\%) & 45900 & (15\%) & 27270 & (9\%)  \\ \hline
         %$29$ &  &  & &  &  &   \\ \hline
    \end{tabular}
\end{table}
\end{center}
%\newpage

% When we take $\alpha_i$'s and $\beta_j$'s to be pairwise distinct elements in $\mathbb{F}_{p} \smallsetminus \{0,1\}$, the number of choices of $(\alpha_1, \{ \{ \alpha_2,\alpha_3 \}, \{ \beta_2, \beta_3 \} \})$ is
% \[
% (p-2) \times \binom{p-3}{4} \times \binom{4}{2} \times \frac{1}{2}= \frac{(p-2)(p-3)(p-4)(p-5)(p-6)}{8}.
% \]

\begin{example}\label{ex:sing}
Let $k=\overline{\mathbb{F}}_{31}$, and let $\tilde{C}:F(Y,Z)=0$ be our sextic curve associated with a point $(\alpha_1,\alpha_2, \alpha_3, \beta_2, \beta_3) \in k^5$ of Howe type.
\begin{enumerate}
    \item[(I)] Put $(\alpha_1,\alpha_2, \alpha_3, \beta_2, \beta_3) =(3,9,27,19,26) \in \mathbb{F}_{31}^5$, the computed sextic is
    \[
    F= 27 Y^6 + 11 Y^4 Z^2 + 18 Y^4 + 25 Y^2 Z^4 + 4 Y^2 + Z^6 + 8 Z^4 + 27 Z^2
    \]
    with $c_{22} = 0$.
    In this case, the projective closure of $\tilde{C}$ has exactly five double points.
    One is $(0:0:1)$, and the others are $(\pm 7 : \pm 14 : 1)$.
    
    \item[(II)] Put $(\alpha_1,\alpha_2, \alpha_3, \beta_2, \beta_3) =(3,25,17,21,18) \in \mathbb{F}_{31}^5$, we can compute $F$ as
    \[
    F = 23 Y^6 + 4 Y^4 Z^2 + 4 Y^4 + 4 Y^2 Z^4 + 19 Y^2 + Z^6 + 22 Z^4 + 12 Z^2.
    \]
    In this case, the projective closure of $\tilde{C}$ has exactly three double points.
    One is $(0:0:1)$, and the others are points at infinity given by $(\pm 13:1:0)$.

    \item[(III)] For $(\alpha_1,\alpha_2, \alpha_3, \beta_2, \beta_3) =(3,2,16,28,9) \in \mathbb{F}_{31}^5$, the computed sextic is
    \[
    F = 23 Y^6 + 4 Y^4 Z^2 + 4 Y^4 + 4 Y^2 Z^4 + 19 Y^2 + Z^6 + 22 Z^4 + 12 Z^2.
    \]
    There are exactly three singularities on the projective closure of $\tilde{C}$:
    One is $(0:0:1)$, and the others are $(0:\pm 8:1)$.
    Each of them is of multiplicity two.

\end{enumerate}
\end{example}

\section{Concluding remarks}
%============================

In this paper, we studied non-hyperelliptic Howe curves of genus five.
Such a curve is defined as the normalization of a fiber product of two hyperelliptic curves of genera $g_1$ and $g_2$ sharing exactly $r$ ramification points in $\mathbb{P}^1$, for $(g_1,g_2,r) = (2,2,4)$ or $(1,1,0)$.
%or two genus-$1$ curves sharing no ramification point.
We focused on the former case, and explicitly constructed a plane sextic model for Howe curves in the case.
Once the ramification points are given, this sextic model can be computed in a constant number of arithmetics in a field over which the ramification points are defined.
The possible number of singularities on the sextic was also determined, and it was also proved that there are $5$ double points generically.

Our sextic model is useful to analyze Howe curves of genus five as plane singular curves, by constructing their function fields.
For example, one can determine the isomorphy of given two such curves, which derives applications to enumerating curves defined over finite fields, e.g., superspecial curves.
The remaining case $(g_1,g_2,r)=(1,1,0)$ together with such an application will be studied in a separated paper~\cite{MK23-2}.

\bibliography{ref}
\bibliographystyle{plain}

\if 0

\fi

\appendix

\section{Details of some computations in Prop.\ \ref{prop:irr}}\label{app:eta}
%===============================================================

%=====================
\subsection{Case (I)}
%=====================
In this case, we have
\[
\begin{aligned}
\eta_1= & \alpha_1 \alpha_2 \alpha_3 \beta_2 + \alpha_1 \alpha_2 \alpha_3 \beta_3 - \alpha_1 \alpha_2 \alpha_3 - \alpha_1 \alpha_2 \beta_2 \beta_3 - \alpha_1 \alpha_3 \beta_2 \beta_3\\
& +\alpha_1 \beta_2 \beta_3 - \alpha_2^2 \alpha_3 \beta_3 + \alpha_2^2 \beta_2 \beta_3 - \alpha_2 \alpha_3 \beta_2^2 + \alpha_2 \alpha_3 \beta_2\\
& + \alpha_2 \alpha_3 \beta_3 - \alpha_2 \beta_2 \beta_3 + \alpha_3 \beta_2^2 \beta_3 - \alpha_3 \beta_2 \beta_3,
\end{aligned}
\]
\[
\begin{aligned}
\eta_2 =& \alpha_1 \alpha_2 \alpha_3 \beta_2 + \alpha_1 \alpha_2 \alpha_3 \beta_3 - \alpha_1 \alpha_2 \alpha_3 - \alpha_1 \alpha_2 \beta_2 \beta_3 - \alpha_1 \alpha_3 \beta_2 \beta_3\\
& + \alpha_1 \beta_2 \beta_3 - \alpha_2^2 \alpha_3 \beta_3 + \alpha_2^2 \alpha_3 + \alpha_2^2 \beta_2 \beta_3 - \alpha_2^2 \beta_2 - \alpha_2 \alpha_3 \beta_2^2 \\
& + \alpha_2 \beta_2^2 + \alpha_3 \beta_2^2 \beta_3 - \beta_2^2 \beta_3,
\end{aligned}
\]
\[
\begin{aligned}
\eta_3 =& \alpha_1 \alpha_2^2 \alpha_3 - \alpha_1 \alpha_2^2 \beta_2 - \alpha_1 \alpha_2 \alpha_3 + \alpha_1 \alpha_2 \beta_2^2 - \alpha_1 \beta_2^2 \beta_3 + \alpha_1 \beta_2 \beta_3\\
& - \alpha_2^2 \alpha_3 \beta_3 + \alpha_2^2 \beta_2 \beta_3 - \alpha_2 \alpha_3 \beta_2^2 + \alpha_2 \alpha_3 \beta_2 + \alpha_2 \alpha_3 \beta_3 \\
& - \alpha_2 \beta_2 \beta_3 + \alpha_3 \beta_2^2 \beta_3 - \alpha_3 \beta_2 \beta_3,
\end{aligned}
\]
\[
\begin{aligned}
\eta_2' = & \alpha_1 \alpha_2 \alpha_3  \beta_2 + \alpha_1 \alpha_2 \alpha_3  \beta_3 - \alpha_1 \alpha_2 \alpha_3 - \alpha_1 \alpha_2  \beta_2  \beta_3 - \alpha_1 \alpha_3  \beta_2  \beta_3\\
& +  \alpha_1 \beta_2 \beta_3 + \alpha_2^2 \alpha_3 \beta_2 - \alpha_2^2 \alpha_3 \beta_3 - \alpha_2^2 \beta_2^2 + \alpha_2^2 \beta_2 \beta_3\\
& - 2 \alpha_2 \alpha_3 \beta_2^2 - \alpha_2 \alpha_3 \beta_2 \beta_3 + \alpha_2 \alpha_3\beta_2 + \alpha_2 \alpha_3 \beta_3 + \alpha_2 \beta_2^3\\
& + \alpha_2 \beta_2^2 \beta_3 - \alpha_2 \beta_2 \beta_3 + 2 \alpha_3 \beta_2^2 \beta_3 - \alpha_3 \beta_2 \beta_3 - \beta_2^3 \beta_3
\end{aligned}
\]
with $\eta_3' =\alpha_2 - \beta_2$, whence
 \[
    \begin{cases}
        \eta_1 - \eta_1' = -\beta_3 (\alpha_3 - \beta_2) (\alpha_2 - \beta_3)(\alpha_2 - \beta_2),\\
        \eta_1 -\eta_2' = -\beta_2 (\alpha_3 - \beta_2) (\alpha_2 - \beta_3)(\alpha_2 - \beta_2),\\
        \eta_2 - \eta_1' = -(\beta_3-1) (\alpha_3 - \beta_2) (\alpha_2 - \beta_3)(\alpha_2-\beta_2),\\
        \eta_2 - \eta_2' = -(\beta_2-1) (\alpha_3 - \beta_2) (\alpha_2 - \beta_3)(\alpha_2-\beta_2),\\
        \eta_3 - \eta_1' = (\alpha_1 - \beta_3)(\alpha_3 - \beta_2)(\alpha_2 - \beta_3) (\alpha_2 - \beta_2),\\
        \eta_3 - \eta_2' = (\alpha_1 - \beta_2)(\alpha_3 - \beta_2) (\alpha_2 - \beta_3)(\alpha_2-\beta_2).
    \end{cases}
\]

%=====================
\subsection{Case (II)}
%=====================

In this case, we have
\[
\begin{aligned}
\eta_1 =& \alpha_1 \alpha_2 \alpha_3 \beta_2 + \alpha_1 \alpha_2 \alpha_3 \beta_3 - \alpha_1 \alpha_2 \alpha_3 - \alpha_1 \alpha_2 \beta_2 \beta_3 - \alpha_1 \alpha_3 \beta_2 \beta_3\\
& + \alpha_1 \beta_2 \beta_3 - \alpha_2^2 \alpha_3 \beta_2 + \alpha_2^2 \beta_2 \beta_3 + \alpha_2 \alpha_3 \beta_2 - \alpha_2 \alpha_3 \beta_3^2 + \alpha_2 \alpha_3 \beta_3\\
& - \alpha_2 \beta_2 \beta_3 + \alpha_3 \beta_2 \beta_3^2 - \alpha_3 \beta_2 \beta_3,
\end{aligned}
\]
\[
\begin{aligned}
\eta_2=& \alpha_1 \alpha_2 \alpha_3 \beta_2 + \alpha_1 \alpha_2 \alpha_3 \beta_3 - \alpha_1 \alpha_2 \alpha_3 - \alpha_1 \alpha_2 \beta_2 \beta_3 - \alpha_1 \alpha_3 \beta_2 \beta_3\\
&+ \alpha_1 \beta_2 \beta_3 - \alpha_2^2 \alpha_3 \beta_2 + \alpha_2^2 \alpha_3 + \alpha_2^2 \beta_2 \beta_3 - \alpha_2^2 \beta_3 - \alpha_2 \alpha_3 \beta_3^2\\
& + \alpha_2 \beta_3^2 + \alpha_3 \beta_2 \beta_3^2 - \beta_2 \beta_3^2,
\end{aligned}
\]
\[
\begin{aligned}
\eta_3 =&\alpha_1 \alpha_2^2 \alpha_3 - \alpha_1 \alpha_2^2 \beta_3 - \alpha_1 \alpha_2 \alpha_3 + \alpha_1 \alpha_2 \beta_3^2 - \alpha_1 \beta_2 \beta_3^2 + \alpha_1 \beta_2 \beta_3\\
& - \alpha_2^2 \alpha_3 \beta_2 + \alpha_2^2 \beta_2 \beta_3 + \alpha_2 \alpha_3 \beta_2 - \alpha_2 \alpha_3 \beta_3^2 + \alpha_2 \alpha_3 \beta_3 - \alpha_2 \beta_2 \beta_3\\
& + \alpha_3 \beta_2 \beta_3^2 - \alpha_3 \beta_2 \beta_3
\end{aligned}
\]
\[
\begin{aligned}
\eta_2' = & \alpha_1 \alpha_2 \alpha_3 \beta_2 + \alpha_1 \alpha_2 \alpha_3 \beta_3 - \alpha_1 \alpha_2 \alpha_3 - \alpha_1 \alpha_2 \beta_2 \beta_3 - \alpha_1 \alpha_3 \beta_2 \beta_3 \\
& + \alpha_1 \beta_2 \beta_3 - \alpha_2^2 \alpha_3 \beta_2 + \alpha_2^2 \alpha_3 \beta_3 + \alpha_2^2 \beta_2 \beta_3 - \alpha_2^2 \beta_3^2 - \alpha_2 \alpha_3 \beta_2 \beta_3\\
& + \alpha_2 \alpha_3 \beta_2 - 2 \alpha_2 \alpha_3 \beta_3^2 + \alpha_2 \alpha_3 \beta_3 + \alpha_2 \beta_2 \beta_3^2 - \alpha_2 \beta_2 \beta_3 + \alpha_2 \beta_3^3 \\
& + 2 \alpha_3 \beta_2 \beta_3^2 - \alpha_3 \beta_2 \beta_3 - \beta_2 \beta_3^3
\end{aligned}
\]
with $\eta_3' =\alpha_2 - \beta_3$, whence
\[
    \begin{cases}
        \eta_1 - \eta_1' = -\beta_2 (\alpha_3 - \beta_3) (\alpha_2 - \beta_3)(\alpha_2 - \beta_2),\\
        \eta_1 -\eta_2' = -\beta_3 (\alpha_3 - \beta_3) (\alpha_2 - \beta_3)(\alpha_2 - \beta_2),\\
        \eta_2 - \eta_1' = -(\beta_2-1) (\alpha_3 - \beta_3) (\alpha_2 - \beta_3)(\alpha_2-\beta_2),\\
        \eta_2 - \eta_2' = -(\beta_3-1) (\alpha_3 - \beta_3) (\alpha_2 - \beta_3)(\alpha_2-\beta_2),\\
        \eta_3 - \eta_1' = (\alpha_1 - \beta_2)(\alpha_3 - \beta_3)(\alpha_2 - \beta_3) (\alpha_2 - \beta_2),\\
        \eta_3 - \eta_2' = (\alpha_1 - \beta_3)(\alpha_3 - \beta_3) (\alpha_2 - \beta_3)(\alpha_2-\beta_2).
    \end{cases}
\]

%=====================
\subsection{Case (III)}
%=====================

In this case, we have
\[
\begin{aligned}
\eta_1=&\alpha_1 \alpha_2 \alpha_3 \beta_2 + \alpha_1 \alpha_2 \alpha_3 \beta_3 - \alpha_1 \alpha_2 \alpha_3 - \alpha_1 \alpha_2 \beta_2 \beta_3 - \alpha_1 \alpha_3 \beta_2 \beta_3\\
& + \alpha_1 \beta_2 \beta_3 - \alpha_2 \alpha_3^2 \beta_3 - \alpha_2 \alpha_3 \beta_2^2 + \alpha_2 \alpha_3 \beta_2 + \alpha_2 \alpha_3 \beta_3\\
& + \alpha_2 \beta_2^2 \beta_3 -
\alpha_2 \beta_2 \beta_3  + \alpha_3^2 \beta_2 \beta_3 - \alpha_3 \beta_2 \beta_3,
\end{aligned}
\]
\[
\begin{aligned}
\eta_2=&\alpha_1 \alpha_2 \alpha_3 \beta_2 + \alpha_1 \alpha_2 \alpha_3 \beta_3 - \alpha_1 \alpha_2 \alpha_3 - \alpha_1 \alpha_2 \beta_2 \beta_3 - \alpha_1 \alpha_3 \beta_2 \beta_3\\
& + \alpha_1 \beta_2 \beta_3  - \alpha_2 \alpha_3^2 \beta_3 + \alpha_2 \alpha_3^2 - \alpha_2 \alpha_3 \beta_2^2 + \alpha_2 \beta_2^2 \beta_3 + \alpha_3^2 \beta_2 \beta_3\\
&  - \alpha_3^2 \beta_2 + \alpha_3 \beta_2^2 - \beta_2^2 \beta_3,
\end{aligned}
\]
\[
\begin{aligned}
\eta_3=&\alpha_1 \alpha_2 \alpha_3^2 - \alpha_1 \alpha_2 \alpha_3 - \alpha_1 \alpha_3^2 \beta_2 + \alpha_1 \alpha_3 \beta_2^2 - \alpha_1 \beta_2^2 \beta_3 + \alpha_1 \beta_2 \beta_3\\
& - \alpha_2 \alpha_3^2 \beta_3 - \alpha_2 \alpha_3 \beta_2^2 + \alpha_2 \alpha_3 \beta_2 + \alpha_2 \alpha_3 \beta_3 + \alpha_2 \beta_2^2 \beta_3\\
& - \alpha_2 \beta_2 \beta_3 + \alpha_3^2 \beta_2 \beta_3 - \alpha_3 \beta_2 \beta_3
\end{aligned}
\]
\[
\begin{aligned}
\eta_2':=&\alpha_1 \alpha_2 \alpha_3 \beta_2 + \alpha_1 \alpha_2 \alpha_3 \beta_3 - \alpha_1 \alpha_2 \alpha_3 - \alpha_1 \alpha_2 \beta_2 \beta_3 - \alpha_1 \alpha_3 \beta_2 \beta_3 \\
& + \alpha_1 \beta_2 \beta_3 + \alpha_2 \alpha_3^2 \beta_2 - \alpha_2 \alpha_3^2 \beta_3 - 2 \alpha_2 \alpha_3 \beta_2^2 - \alpha_2 \alpha_3 \beta_2 \beta_3\\
& + \alpha_2 \alpha_3 \beta_2 + \alpha_2 \alpha_3 \beta_3   + 2 \alpha_2 \beta_2^2 \beta_3 - \alpha_2 \beta_2 \beta_3 - \alpha_3^2 \beta_2^2 \\
& + \alpha_3^2 \beta_2 \beta_3 + \alpha_3 \beta_2^3 + \alpha_3 \beta_2^2 \beta_3 - \alpha_3 \beta_2 \beta_3 - \beta_2^3 \beta_3
\end{aligned}
\]
with $\eta_3' =\alpha_3 - \beta_2$, whence
\[
\begin{cases}
    \eta_1 - \eta_1' = -\beta_3 (\alpha_3 - \beta_3) (\alpha_3 - \beta_2)(\alpha_2 - \beta_2),\\
    \eta_1 -\eta_2' = -\beta_2 (\alpha_3 - \beta_3) (\alpha_3 - \beta_2)(\alpha_2 - \beta_2),\\
    \eta_2 - \eta_1' = -(\beta_3-1) (\alpha_3 - \beta_3) (\alpha_3 - \beta_2)(\alpha_2-\beta_2),\\
    \eta_2 - \eta_2' = -(\beta_2-1) (\alpha_3 - \beta_3) (\alpha_3 - \beta_2)(\alpha_2-\beta_2),\\
    \eta_3 - \eta_1' = (\alpha_1 - \beta_3)(\alpha_3 - \beta_3)(\alpha_3 - \beta_2) (\alpha_2 - \beta_2),\\
    \eta_3 - \eta_2' = (\alpha_1 - \beta_2)(\alpha_3 - \beta_3) (\alpha_3 - \beta_2)(\alpha_2-\beta_2).
\end{cases}
\]

%=====================
\subsection{Case (IV)}
%=====================

In this case, we have
\[
\begin{aligned}
    \eta_1 = & \; \alpha_1  \alpha_2  \alpha_3 \beta_2 +  \alpha_1 \alpha_2 \alpha_3 \beta_3 - \alpha_1 \alpha_2 \alpha_3 - \alpha_1 \alpha_2 \beta_2 \beta_3 - \alpha_1 \alpha_3 \beta_2 \beta_3\\
    & + \alpha_1 \beta_2 \beta_3 - \alpha_2 \alpha_3^2 \beta_2 + \alpha_2 \alpha_3 \beta_2 - \alpha_2 \alpha_3 \beta_3^2 + \alpha_2 \alpha_3 \beta_3 + \alpha_2 \beta_2 \beta_3^2\\
    & - \alpha_2 \beta_2 \beta_3 + \alpha_3^2 \beta_2 \beta_3 - \alpha_3 \beta_2 \beta_3,
\end{aligned}
\]
\[
\begin{aligned}
    \eta_2 =& \; \alpha_1 \alpha_2 \alpha_3 \beta_2 + \alpha_1 \alpha_2 \alpha_3 \beta_3 - \alpha_1 \alpha_2 \alpha_3 - \alpha_1 \alpha_2 \beta_2 \beta_3 - \alpha_1 \alpha_3 \beta_2 \beta_3\\
    & + \alpha_1 \beta_2 \beta_3 - \alpha_2 \alpha_3^2 \beta_2 + \alpha_2 \alpha_3^2 - \alpha_2 \alpha_3 \beta_3^2 + \alpha_2 \beta_2 \beta_3^2 + \alpha_3^2 \beta_2 \beta_3\\
    & - \alpha_3^2 \beta_3 + \alpha_3 \beta_3^2 - \beta_2 \beta_3^2,
\end{aligned}
\]
\[
\begin{aligned}
    \eta_3 =& \; \alpha_1 \alpha_2 \alpha_3^2 - \alpha_1 \alpha_2 \alpha_3 - \alpha_1 \alpha_3^2 \beta_3 + \alpha_1 \alpha_3 \beta_3^2  - \alpha_1 \beta_2 \beta_3^2 + \alpha_1 \beta_2 \beta_3\\
    & - \alpha_2 \alpha_3^2 \beta_2 + \alpha_2 \alpha_3 \beta_2 - \alpha_2 \alpha_3 \beta_3^2 + \alpha_2 \alpha_3 \beta_3 + \alpha_2 \beta_2 \beta_3^2 - \alpha_2 \beta_2 \beta_3\\
    & + \alpha_3^2 \beta_2 \beta_3 - \alpha_3 \beta_2 \beta_3
    % D_{\xi_1}/G_{\xi_1} =& \; \alpha_1 \alpha_2 \alpha_3^2 \beta_3^2 - \alpha_1 \alpha_2 \alpha_3^2 \beta_3 - 2 \alpha_1 \alpha_2 \alpha_3 \beta_2 \beta_3^2 + 2 \alpha_1 \alpha_2 \alpha_3 \beta_2 \beta_3 \\
    % & + \alpha_1 \alpha_2 \beta_2^2 \beta_3^2 - \alpha_1 \alpha_2 \beta_2^2 \beta_3 - \alpha_1 \alpha_3^2 \beta_3^3 + \alpha_1 \alpha_3^2 \beta_3^2 + 2 \alpha_1 \alpha_3 \beta_2 \beta_3^3\\
    % & - 2 \alpha_1 \alpha_3 \beta_2 \beta_3^2 - \alpha_1 \beta_2^2 \beta_3^3 + \alpha_1 \beta_2^2 \beta_3^2 - \alpha_2 \alpha_3^2 \beta_3^3 + \alpha_2 \alpha_3^2 \beta_3^2\\
    % & + 2 \alpha_2 \alpha_3 \beta_2 \beta_3^3 -2 \alpha_2 \alpha_3 \beta_2 \beta_3^2 - \alpha_2 \beta_2^2 \beta_3^3+ \alpha_2 \beta_2^2 \beta_3^2 + \alpha_3^2 \beta_3^4\\
    % & - \alpha_3^2 \beta_3^3 - 2 \alpha_3 \beta_2 \beta_3^4 + 2 \alpha_3 \beta_2 \beta_3^3 + \beta_2^2 \beta_3^4 - \beta_2^2 \beta_3^3,\\
    \end{aligned}
    \]
    \[
    \begin{aligned}
    \eta_2' =&\; \alpha_1 \alpha_2 \alpha_3 \beta_2 + \alpha_1 \alpha_2 \alpha_3 \beta_3 - \alpha_1 \alpha_2 \alpha_3 - \alpha_1 \alpha_2 \beta_2 \beta_3 - \alpha_1 \alpha_3 \beta_2 \beta_3\\
    & + \alpha_1 \beta_2 \beta_3 - \alpha_2 \alpha_3^2 \beta_2 + \alpha_2 \alpha_3^2 \beta_3 - \alpha_2 \alpha_3 \beta_2 \beta_3 + \alpha_2 \alpha_3 \beta_2 -2 \alpha_2 \alpha_3 \beta_3^2\\
    & + \alpha_2 \alpha_3 \beta_3  + 2 \alpha_2 \beta_2 \beta_3^2- \alpha_2 \beta_2 \beta_3 + \alpha_3^2 \beta_2 \beta_3 - \alpha_3^2 \beta_3^2 + \alpha_3 \beta_2 \beta_3^2\\
    & - \alpha_3 \beta_2 \beta_3 + \alpha_3 \beta_3^3 - \beta_2 \beta_3^3,
    % D_{\xi_2}/G_{\xi_2} =&\; \alpha_1 \alpha_3 \beta_3^2 - \alpha_1 \alpha_3 \beta_3 - \alpha_1 \beta_2 \beta_3^2 + \alpha_1 \beta_2 \beta_3\\
    % &- \alpha_3 \beta_3^3 + \alpha_3 \beta_3^2 + \beta_2 \beta_3^3 - \beta_2 \beta_3^2,
\end{aligned}
\]
with $\eta_3' =\alpha_3 - \beta_3$, whence
\[
\begin{cases}
    \eta_1 - \eta_1' = - \beta_2 (\alpha_3 - \beta_3) (\alpha_3 - \beta_2)(\alpha_2 - \beta_3),\\
    \eta_1 -\eta_2' = - \beta_3 (\alpha_3 - \beta_3) (\alpha_3 - \beta_2)(\alpha_2 - \beta_3),\\
    \eta_2 - \eta_1' = -(\beta_2-1) (\alpha_3 - \beta_3) (\alpha_3 - \beta_2)(\alpha_2-\beta_3),\\
    \eta_2 - \eta_2' = -(\beta_3-1) (\alpha_3 - \beta_3) (\alpha_3 - \beta_2)(\alpha_2-\beta_3),\\
    \eta_3 - \eta_1' = (\alpha_1 - \beta_2)(\alpha_3 - \beta_3) (\alpha_3 - \beta_2)(\alpha_2-\beta_3),\\
    \eta_3 - \eta_2' = (\alpha_1 - \beta_3)(\alpha_3 - \beta_3) (\alpha_3 - \beta_2)(\alpha_2-\beta_3).
\end{cases}
\]

\end{document}